\documentclass[a4paper,11pt]{amsart}

\usepackage{mathpazo}

\usepackage{pgf}
\usepackage{tikz}
\usepackage[all]{xy}
\SelectTips{cm}{}
\usepackage{tabularx}
\usepackage[pdftex,pagebackref,
    ,pdfborder={0 0 0}
    ,urlcolor=black,a4paper,hypertexnames=false]{hyperref}
\hypersetup{pdfauthor=Diarmuid Crowley and Clara
  L\"oh,pdftitle=Functorial semi-norms on singular homology and
  (in)flexible manifolds}

\def\enumref#1{(\ref{#1})}

\newtheorem{thm}{Theorem}[section]
\newtheorem{prop}[thm]{Proposition}
\newtheorem{lem}[thm]{Lemma}
\newtheorem{cor}[thm]{Corollary}

\newtheorem{question}[thm]{Question}

\theoremstyle{remark}
\newtheorem{rem}[thm]{Remark}
\newtheorem{exa}[thm]{Example}
\newtheorem{nonexa}[thm]{(Non-)Example}

\theoremstyle{definition}
\newtheorem{defi}[thm]{Definition}

\usepackage{amsfonts}
\newcommand{\Z}{\mathbb{Z}}
\newcommand{\Q}{\mathbb{Q}}
\newcommand{\R}{\mathbb{R}}
\newcommand{\N}{\mathbb{N}}
\newcommand{\C}{\mathbb{C}}
\newcommand{\tensor}{\otimes}

\newcommand{\CW}{{CW}}

\newcommand{\fclr}[1]{[#1]_\R}
\newcommand{\fclq}[1]{[#1]_\Q}
\newcommand{\args}{\,\cdot\,}
\DeclareMathOperator{\Mfd}{Mfd}
\DeclareMathOperator{\id}{id}
\DeclareMathOperator{\BO}{BO}
\DeclareMathOperator{\GL}{GL}
\DeclareMathOperator{\BSO}{BSO}
\DeclareMathOperator{\SO}{SO}
\DeclareMathOperator{\rank}{rank}
\DeclareMathOperator{\Hom}{Hom}
\DeclareMathOperator{\sign}{sign}
\DeclareMathOperator{\Map}{Map}
\DeclareMathOperator{\vol}{vol}

\makeatletter
\newcommand\norm{\bBigg@{0.6}}

\makeatother
\newcommand{\fsn}[2][norm]{\csname #1l\endcsname\bracevert\!#2\!%
                           \csname #1r\endcsname\bracevert}
\newcommand{\realmfd}[2][norm]{{\mathcal M}\csname #1l\endcsname(#2%
                                           \csname #1r\endcsname)}

\newcommand{\csum}{\mathbin{\#}}

\def\to{\longrightarrow}

\def\epsilon{\varepsilon}

\usepackage{color}
\usepackage{pdfcolmk}

\author{Diarmuid Crowley}
\author{Clara L\"oh}
\title[Functorial semi-norms and inflexible manifolds]{Functorial semi-norms on singular
  homology\\
  and (in)flexible manifolds}
\date{\today.\ \copyright{\ D.~Crowley and C.~L\"oh 2011}}

\begin{document}

\begin{abstract}
  A functorial semi-norm on singular homology is a collection of
  semi-norms on the singular homology groups of spaces such that continuous maps
  between spaces induce norm-decreasing maps in homology. Functorial
  semi-norms can be used to give constraints on the possible mapping
  degrees of maps between oriented manifolds.

  In this paper, we use information about the degrees of maps between
  manifolds to construct new functorial semi-norms with interesting
  properties. In particular, we answer a question of Gromov by
  providing a functorial semi-norm that takes finite positive values
  on homology classes of certain \emph{simply connected} spaces. Our
  construction relies on the existence of simply connected manifolds
  that are \emph{inflexible} in the sense that all their self-maps
  have degree~$-1$,~$0$, or~$1$. The existence of such manifolds was
  first established by Arkowitz and Lupton; we extend their methods to
  produce a wide variety of such manifolds.
\end{abstract}

\maketitle

\section{Introduction}

Enriching algebraic invariants with metric data is a common theme in
many branches of mathematics. Gromov introduced the concept of
functorial semi-norms on singular
homology~\cite[Section~5.34]{gromov}, which are an example of this
paradigm in topology.

A \emph{functorial semi-norm} on singular homology consists of the
addition of a semi-normed structure to the singular homology groups
with $\R$-coefficients in such a way that continuous maps induce
linear maps on homology of norm at most~$1$
(Definition~\ref{def:funcsn}).  An interesting aspect is that suitable
functorial semi-norms give a systematic way to deduce degree theorems
for maps between manifolds (Remark~\ref{rem:degthms}). Conversely, in
the present paper, we translate knowledge about degrees of maps
between manifolds to construct new functorial semi-norms.

A central example of a functorial semi-norm on singular homology,
studied by Gromov~\cite{vbc}, is the $\smash{\ell^1}$-semi-norm given
by taking the infimum of the $\smash{\ell^1}$-norms of all cycles
representing a given homology class
(Example~\ref{exa:l1seminorm}). The $\smash{\ell^1}$-semi-norm gives
rise to lower bounds for the minimal volume and hence leads to
interesting applications in Riemannian geometry~\cite{vbc}. On the
other hand, using bounded cohomology, Gromov showed that the
\mbox{$\ell^1$-se}\-mi-norm vanishes on classes of non-zero degree of
simply connected spaces~\cite{vbc}, and later raised the 
question whether every functorial semi-norm on singular homology in
non-zero degree is trivial on all simply connected
spaces~~\cite[Remark~(b) in~5.35]{gromov}. More precisely, we
formulate this problem as follows:

\begin{question}
  \label{q:gromov}
  Let $d \in \N_{>0}$.
  \begin{enumerate}
    \item\label{enum:infinite} Does every (possibly infinite) functorial semi-norm on
      singular homology in degree~$d$ take only the values~$0$
      and~$\infty$ on homology classes of simply connected spaces?
    \item\label{enum:finite} Does every finite functorial semi-norm on singular homology
      in degree~$d$ vanish on homology classes of simply connected
      spaces? 
  \end{enumerate}
\end{question}

In this paper, we answer the first part of this question in the negative
(Corollary~\ref{cor:inflexiblesn}):

\begin{thm}\label{thm:inflexiblesnintro}
  There are functorial semi-norms on singular homology that are
  positive and finite on certain homology classes of simply connected
  spaces.
\end{thm}

More concretely, we give examples of such functorial semi-norms in all
degrees in the set~$\{64\} \cup \{ d \cdot k \mid k \in \N_{>0},\ d
\in \{108,208,228\}\}$ (Corollary~\ref{cor:manyinflexible}). 

On the other hand, we give a positive answer to
Question~\ref{q:gromov}\enumref{enum:finite} in low dimensions
(Section~\ref{subsec:positiveanswer}):

\begin{thm}\label{thm:fsnlowdimintro}
  All finite functorial semi-norms on singular homology in the
  degrees~$1, \dots, 6$ vanish on all homology clases of simply
  connected spaces.
\end{thm}

The key to proving Theorem~\ref{thm:inflexiblesnintro}
and~\ref{thm:fsnlowdimintro} is to gain an understanding of the class of simply
connected inflexible manifolds.

\begin{defi}[Inflexible manifolds]
  If $M$ and $N$ are oriented closed connected manifolds of the same
  dimension, then we write
  \[ \deg(N,M) := \{ \deg f \mid \text{$f \colon N \longrightarrow M$ continuous} \} 
  \]
  for the set of all possible mapping degrees for maps from~$N$
  to~$M$. An oriented closed connected manifold~$M$ is \emph{inflexible} if 
  $\deg(M,M) \subset \{-1,0,1\}. 
  $
\end{defi}

The proof of Theorem~\ref{thm:inflexiblesnintro} consists of two
main steps:
\begin{itemize}
  \item \emph{Generating functorial semi-norms via manifolds.}  Using
    the fact that singular homology classes can (up to a scalar
    multiple) be represented by fundamental classes of oriented closed
    connected manifolds (Section~\ref{sec:repmfd}), we show how
    functorial semi-norms on fundamental classes of manifolds of a
    given dimension can be extended to functorial semi-norms on
    singular homology (Theorem~\ref{thm:genfunsn}).
  \item \emph{Inflexible manifolds.}  With the help of simply
    connected inflexible manifolds, we construct a functorial
    semi-norm on fundamental classes of manifolds that is positive and
    finite on the given simply connected inflexible manifold
    (Corollary~\ref{cor:inflexiblesn}). 
\end{itemize}

Simply connected inflexible manifolds can be constructed by means of
rational homotopy theory and surgery theory. The first examples of
such manifolds were given by Arkowitz and Lupton~\cite[Examples~5.1
and~5.2]{arkowitzlupton}; these examples have dimension~$208$ and
$228$ respectively. Using and extending the methods of Arkowitz and
Lupton, we give more examples of simply connected inflexible
manifolds: For instance, we have examples in dimension~$64$ (the
smallest dimension known before being~$208$) and~$108$. Starting from
these basic examples, we can construct many more simply connected
inflexible manifolds:
\begin{itemize}
  \item In general, it is not clear that connected sums and products
    of inflexible manifolds are inflexible; however, in certain cases
    this is true (Section~\ref{subsec:sum}
    and~\ref{subsec:product}). This provides in infinitely many
    dimensions infinitely many rational homotopy types of oriented
    closed simply connected inflexible manifolds
    (Corollary~\ref{cor:manyinflexible}). 
  \item In addition, using scaling of the fundamental class with
    respect to a rationalisation, we obtain infinitely many homotopy
    types of oriented closed simply connected inflexible manifolds within
    the same rational homotopy type (Proposition~\ref{prop:scaling}).
  \item Moreover, we can show that for manifolds being simply
    connected and inflexible is generic in the sense that in
    infinitely many dimensions every rational bordism class is
    represented by a simply connected inflexible manifold
    (Proposition~\ref{prop:inflexibleandrationalbordism}).
  \item Also, there are simply connected inflexible smooth manifolds
    satisfying certain tangential structure constraints such as being
    stably parallelisable or non-spinable
    (Section~\ref{subsec:moreinflexiblemanifolds}). 
\end{itemize}

However, from our construction it is not clear whether the examples
from Theorem~\ref{thm:inflexiblesnintro} are finite functorial
semi-norms; so Gromov's question remains open for finite functorial
semi-norms in degree~$7$ and higher.  More precisely, we prove the
following proposition (Proposition~\ref{prop:equivstronglyinflexible})
where an oriented closed connected $n$-manifold~$M$ is called
\emph{strongly inflexible} if for any oriented closed connected
$n$-manifold~$N$ the set~$\deg(N, M)$ is finite (Definition~\ref{def:stronglyinflexible}):

\begin{prop}
  For $d \in \N_{\geq 4}$ the following statements are equivalent: 
  \begin{enumerate}
   \item There is a \emph{finite} functorial semi-norm~$\fsn \args$
     on~$H_d(\args;\R)$ such that for some homology class~$\alpha \in
     H_d(X;\R)$ of some simply connected space~$X$ we have $\fsn
     \alpha \neq 0$.
   \item There exists an oriented closed simply connected $d$-manifold
     that is strongly inflexible.
  \end{enumerate}
\end{prop}

No example of a simply connected strongly inflexible manifold seems to
be known to date: if such a manifold exists, it has dimension at least~$7$.

\begin{rem}
  Since this paper was posted Costoya and Viruel~\cite{costoyaviruel} and also
  Amann~\cite{amann} have further extended the list of examples and
  constructions of simply connected inflexible manifolds.  Amann~\cite{amann}
  has also given new examples of simply connected flexible manifolds.
  \end{rem}

\subsection*{Organisation of this paper}
We start by giving an introduction to functorial semi-norms
(Section~\ref{sec:fsn}). In Section~\ref{sec:repmfd} we recall Thom's
result on representation of homology classes by fundamental classes of
manifolds, which is the key ingredient for generating functorial
semi-norms via functorial semi-norms for manifolds
(Section~\ref{sec:genfuncsn}). We discuss the relationship between
functorial semi-norms on the singular chain complex and functorial
semi-norms on singular homology in Section~\ref{sec:chainfsn}. In
Section~\ref{sec:fsnsimplyconnected} we prove the
Theorems~\ref{thm:inflexiblesnintro} and~\ref{thm:fsnlowdimintro}. The
proof of Theorem~\ref{thm:inflexiblesnintro} is based on the
construction of simply connected inflexible manifolds; we carefully
review and extend the construction of Arkowitz and Lupton of simply
connected inflexible manifolds in
Section~\ref{sec:(in)flexiblemanifolds}, the technical aspects being
deferred to Section~\ref{app:I}. Finally, Section~\ref{app:II}
contains the study of inheritance properties of being inflexible and
evidence for the genericity of inflexibility in the class of simply
connected manifolds.

\subsection*{Acknowledgements}
We are indebted to Donald Stanley who drew our attention to the
examples of Arkowitz and Lupton.  Moreover, we would like to thank
Thomas Schick for interesting discussions.  We are grateful to
Jonathan Bowden for pointing out a mistake in a previous version.
Part of this work was supported by the HIM~trimester program
\emph{Rigidity} and by the SFB~878 \emph{Groups, Geometry and
  Actions}.

\section{Functorial semi-norms}\label{sec:fsn}

Functorial semi-norms assign a notion of ``size'' to singular homology
classes in a functorial way (Definition~\ref{def:funcsn}).

In this paper, we use the following convention: A \emph{semi-norm}
on an $\R$-vector space~$V$ is a function~$\fsn{\cdot} \colon V
\longrightarrow [0,\infty]$ satisfying the following properties:
\begin{itemize}
  \item We have $\fsn{0} = 0$.
  \item For all~$x \in V$ and all~$a \in \R\setminus\{0\}$, we have~$\fsn{a\cdot x} =
    |a| \cdot \fsn{x}$, where $|a| \cdot \infty := \infty$.
  \item For all~$x$,~$y \in V$ the triangle inequality~$\fsn{x+y} \leq
    \fsn{x} + \fsn{y}$ holds.
\end{itemize}
A semi-norm is called \emph{finite} if it does not take the
value~$\infty$.

\begin{defi}[Functorial semi-norms~{\cite[Section~5.34]{gromov}}]
  \label{def:funcsn}
  Let $d\in \N$. A \emph{functorial semi-norm (on singular homology)
    in degree~$d$} consists of a choice of a semi-norm~$\fsn{\cdot}$
  on~$H_d(X;\R)$ for any topological space~$X$ such that the
  following ``functoriality'' holds: for all continuous maps~$f \colon
  X \longrightarrow Y$ between topological spaces and all~$\alpha
  \in H_d(X;\R)$ we have
  \[ \fsn[big]{H_d(f;\R) (\alpha)}
     \leq \fsn\alpha .
  \]
  Such a functorial semi-norm is called \emph{finite}, if all the
  semi-norms involved are finite semi-norms.
\end{defi}



\begin{exa}[$\ell^1$-Semi-norm]\label{exa:l1seminorm}
  For a topological space~$X$ let $|\cdot|_1$ denote the $\ell^1$-norm
  on the singular chain complex~$C_*(X;\R)$ with respect to the
  (unordered) basis given by all singular simplices: if~$c=
  \sum_{j=1}^k a_j \cdot \sigma_j \in C_*(X;\R)$ is in reduced form,
  then we define
  \[ |c|_1 := \sum_{j=1}^k |a_j|. 
  \]
  This norm induces a finite semi-norm~$\|\cdot\|_1$, the
  so-called~\emph{$\ell^1$-semi-norm}, on singular homology as
  follows: for all~$\alpha \in H_*(X;\R)$ we set
  \[ \|\alpha\|_1
     := \inf \bigl\{ |c|_1
             \bigm| \text{$c \in C_*(X;\R)$ is a cycle representing~$\alpha$}
             \bigr\}.
  \]
  Looking at the definition of the homomorphisms induced by continuous
  maps in singular homology, it is immediate that $\|\cdot\|_1$ is a
  functorial semi-norm on singular homology.

  An interesting topological invariant derived from the
  $\ell^1$-semi-norm in singular homology is the simplicial volume,
  introduced by Gromov~\cite{gromov}: If $M$ is an oriented closed
  connected manifold, then
  \[ \| M \| := \bigl\| \fclr M \bigr\|_1 \in \R_{\geq 0}
  \]
  is the \emph{simplicial volume} of~$M$, where $\fclr M \in H_{\dim
    M}(M;\R)$ denotes the $\R$-fundamental class of~$M$. E.g., using
  self-maps of non-trivial degree, one sees that the simplicial volume
  of spheres (of non-zero dimension) is zero. On the other hand, for
  example, the simplicial volume of oriented closed connected
  hyperbolic manifolds is non-zero~\cite[Section~0.3, Theorem~6.2]{vbc,thurston}, leading to
  interesting applications in Riemannian geometry~\cite{vbc}.

  The $\ell^1$-semi-norm on singular homology can also be expressed in
  terms of bounded cohomology~\cite[p.~17, Proposition~F.2.2]{vbc,
    benedettipetronio}. Using bounded cohomology, Gromov discovered
  that the $\ell^1$-semi-norm of simply connected spaces is
  trivial~\cite[Section~3.1, Theorem~2.4]{vbc, ivanov}, and, more
  generally, that continuous maps that induce an isomorphism on the
  level of fundamental groups induce norm-preserving maps on the level
  of singular homology~\cite[Section~3.1, Theorem~4.3]{vbc, ivanov}.
\end{exa}

It is tempting to analogously consider $\ell^p$-norms with~$p > 1$;
however, it can be shown that the corresponding definition then leads
to the zero semi-norm on homology in positive degrees (this follows
from an argument similar to (Non-)Example~\ref{nonexa:lp})

\begin{exa}[Domination by products of surfaces]\label{exa:ssn}
  For~$d \in \N$, we define the functorial semi-norm~$\fsn{\cdot}_{S}$
  in degree~$2d$ as follows~\cite[Section~5.34]{gromov}: Let $X$ be a
  topological space, and let $\alpha \in H_{2d}(X;\R)$. Then
  \begin{align*} 
    \fsn \alpha _S
         := \inf \biggl\{ \sum_{j = 1}^k |a_j| \cdot |\chi(S_j)|
            \biggm|\; & k \in \N,\  
                      a_1, \dots, a_k \in \R\setminus\{0\},\\
                    & \text{$S_1, \dots, S_k$ are $d$-fold products}\\ 
                    & \text{\phantom{$S_1$~~} of oriented
                    closed connected surfaces,} \\
                    & \text{$f_1 \colon S_1 \rightarrow X, \dots, f_k
                      \colon S_k \rightarrow X$ continuous}\\
                    & \text{with~$\sum_{j=1}^k a_j \cdot H_{2\cdot d}(f_j;\R) \fclr{S_j} = \alpha$}
            \biggr\}.
  \end{align*}
  In other words, $\fsn\cdot _S$ measures the size of homology classes
  in terms of products of surfaces. In general, this functorial
  semi-norm is \emph{not} finite~\cite{kotschickloeh} because not
  every homology class in even degree can be represented by a product
  of surfaces.  
\end{exa}

\begin{prop}[The surface semi-norm is the $\ell^1$-semi-norm in
  degree~$2$]
  \label{prop:surfacel1fsns}
  Let $X$ be a topological space, and let $\alpha \in H_2(X;\R)$. Then 
  \[ \|\alpha\|_1 = 2 \cdot \fsn\alpha_S.
  \]
\end{prop}
\begin{proof}
  This follows from a result of Barge and
  Ghys~\cite[Proposition~1.9]{bargeghys} (notice that their
  argument applies only to classes that do not need summands
  represented by~$S^2$~\cite[proof of Lemme~1.7]{bargeghys}; 
  however, we can safely ignore these summands as they do not
  contribute to the $\ell^1$-semi-norm).
\end{proof}

\begin{question}\label{q:ssn}
  Does the surface semi-norm vanish on homology classes (of non-zero
  degree) of simply-connected spaces?
\end{question}

Classical arguments from algebraic topology show that this is indeed
true in degrees $2$ and $4$ (see Proposition~\ref{prop:ssnlowdeg}); however, the
question is open in high degrees.

Similarly to the surface semi-norm~$\fsn\cdot _S$, we can also define
functorial semi-norms by looking at domination by, e.g., hyperbolic manifolds
(Example~\ref{exa:hypfsn}).

An interesting aspect of functorial semi-norms is that suitable
functorial semi-norms give a systematic way to deduce degree theorems
for maps between manifolds:  

\begin{rem}[Degree theorems]\label{rem:degthms}
  If $\fsn\cdot$ is a functorial semi-norm on singular homology,
  then by definition we have for all continuous maps~$f \colon M
  \longrightarrow N$ of oriented closed connected manifolds of the
  same dimension the estimate
  \[ |\deg f| \cdot \fsn[big]{\fclr N}
     \leq \fsn[big]{\fclr M};
  \]
  hence, if
  $\fsn{\fclr N} \neq 0$, then we obtain the restriction
  \[ |\deg f| 
     \leq 
     \frac{\fsn[big]{\fclr M}}
          {\fsn[big]{\fclr N}}
  \]
  on the mapping degree. Such restrictions are particularly
  interesting when there are -- at least for certain classes of
  (Riemannian) manifolds -- estimates of~$\fsn{\fclr{\args}}$ in terms
  of the Riemannian volume or other geometric invariants.

  For example, powerful degree theorems have been obtained by the use
  of simplicial volume and its variations~\cite[Section~0.5,
    Section~1.2]{vbc,loehsauer}.
\end{rem}

Conversely, in the following sections, we will translate knowledge
about mapping degrees into constructions of functorial semi-norms with
specific properties.


\section{Representing homology classes by manifolds}\label{sec:repmfd}
As mentioned in the introduction, one of our main tools is to
represent singular homology classes by manifolds. For the sake of
completeness, we recall the following classical result:

\begin{thm}\label{thm:rephomology}
  Let $X$ be a connected \CW-complex, let $d \in \N$ and let $\alpha \in H_d(X;\Q)$
  be a singular homology class. 
  \begin{enumerate}
    \item Then there exists an~$a \in \Q\setminus\{0\}$ and an
      oriented closed connected \mbox{$d$-di}\-men\-sion\-al smooth
      manifold~$M$ together with a continuous map~$f \colon M
      \longrightarrow X$ such that
      \[ a \cdot H_d(f;\Q)\fclq M  = \alpha, 
      \]
      where $\fclq M \in H_d(M;\Q)$ is the rational fundamental class
      of~$M$.
    \item If $X$ is homotopy equivalent to a \CW-complex with finite
      $2$-skeleton and $d \geq 4$, then there exists an~$a \in
      \Q\setminus\{0\}$ and an oriented closed connected 
      $d$-dimensional manifold~$M$ together with a continuous map~$f
      \colon M \longrightarrow X$ such that
      \[ a \cdot H_d(f;\Q)\fclq M  = \alpha 
      \]
      and such that in addition $\pi_1(f) \colon \pi_1(M)
      \longrightarrow \pi_1(X)$ is an isomorphism.
  \end{enumerate}
\end{thm}
\begin{proof}
  The first part is a classical result by
  Thom~\cite{thom}. 

  For the second statement, we apply surgery theory as in
  \cite{kreck}.  Using the notation of \emph{loc.\ cit.}, 
  let $B := X
  \times \BSO$ where $\BSO$ is the classifying space of the stable
  special orthogonal group and let $B \longrightarrow \BO$ be the
  fibration given by projection to~$\BSO$ and the canonical
  covering~$\BSO \longrightarrow \BO$.

  Given an oriented closed connected smooth manifold~$\bar \nu \colon
  M \to \BSO$ and a map $f \colon M \longrightarrow X$, we obtain the
  $B$-manifold $f \times \bar \nu \colon M \longrightarrow X \times
  \BSO =B$.  Hence,
  there is an oriented bordism~$F \colon W \longrightarrow X$ over~$X$
  from~$f \colon M \longrightarrow X$ to a map~$g \colon N
  \longrightarrow X$ such that $g$ is a
  \mbox{$2$-equi}\-valence~\cite[Proposition 4]{kreck}; in particular,
  $g$ induces an isomorphism on fundamental groups. A straightforward
  computation in singular homology shows that
  \[ H_*(g;\Q) \fclq N = H_*(f;\Q)\fclq M - H_*(F;\Q)[\partial W]_\Q
                       = H_*(f;\Q)\fclq M; 
  \]
  choosing $f$ as provided by part~(1) finishes the proof.
\end{proof}

We next extend Theorem \ref{thm:rephomology} to general path-connected
spaces and to homology classes in~$H_*(\args;\R)$ which lie in the
image of the change of coefficients homomorphism~$H_*(\args;\Q)
\longrightarrow H_*(\args;\R)$.  Such classes are called
\emph{rational}, and by the universal coefficients theorem, every
class in~$H_*(\args;\R)$ is a finite $\R$-linear combination of
rational classes.

\begin{cor}\label{cor:repspacesr}
  Let $X$ be a path-connected topological space, let $d \in \N$, and let 
  $\alpha \in H_d(X;\R)$ be rational. 
  \begin{enumerate}
    \item Then there exists an $a \in \Q \setminus \{0\}$ and an
      oriented closed connected smooth $d$-manifold~$M$ together with
      a continuous map~$f \colon M \longrightarrow X$ such that
      $a \cdot H_d(f;\R)\fclr M = \alpha$,
      where $\fclr M \in H_d(M;\R)$ is the real fundamental class of~$M$. 
    \item If $X$ is simply connected and $d \geq 4$, then there is
      an~$a \in \Q \setminus\{0\}$, an oriented closed simply
      connected smooth $d$-manifold~$M$, and a continuous map~$f \colon M
      \longrightarrow X$ with~$a \cdot H_d(f;\R)\fclr M = \alpha$.
  \end{enumerate}
\end{cor}
\begin{proof}
  Out of the combinatorial data of a singular cycle in~$C_d(X;\R)$
  representing~$\alpha$ we can construct a connected finite
  \CW-complex~$X'$, a rational homology class~$\alpha' \in H_d(X';\R)$
  and a continuous map~$f' \colon X' \longrightarrow X$ such that
  $H_d(f';\R)(\alpha') = \alpha$; if $X$ is simply connected, then we
  can also assume that $X'$ is simply connected.  Now the claim easily
  follows from the universal coefficient theorem and the previous
  theorem.
\end{proof}

\section[Generating functorial semi-norms]{Generating functorial semi-norms via special spaces}
\label{sec:genfuncsn}

Every functorial semi-norm on singular homology induces by restriction
a functorial semi-norm on the top homology of oriented closed
connected manifolds. Conversely, examples of functorial semi-norms on
singular homology can be generated by extending functorial semi-norms
on the top homology of oriented closed connected manifolds (of a
given dimension):

\begin{defi}[Associated semi-norm]\label{def:assocfunsn}
  Let $d \in \N$, let $\Mfd_d$ denote the class of all oriented
  closed connected $d$-manifolds, and let $S \subset \Mfd_d$ be a subclass.  
  \begin{itemize}
    \item
      A \emph{functorial semi-norm on 
        fundamental classes of oriented closed connected $d$-manifolds in~$S$}, or briefly a 
        \emph{functorial $S$-semi-norm}, is
      a map $v \colon S \longrightarrow [0,\infty]$ such that
      \[ |\deg f| \cdot v(N) \leq v(M) 
      \]
      holds for all continuous maps~$f \colon M \longrightarrow N$
      with~$N$,~$M \in S$.

      If $S = \Mfd_d$, then we call such a~$v$ a \emph{functorial semi-norm
        on fundamental classes of oriented closed connected
        $d$-manifolds}, briefly a \emph{functorial $\Mfd_d$-semi-norm}.
    \item
      Let $v$ be an $S$-functorial semi-norm. 
      The
      \emph{associated semi-norm}~\mbox{$\fsn{\cdot}$} on singular homology in
      degree~$d$ is defined as follows: For a topological space~$X$ and a
      homology class~$\alpha \in H_d(X;\R)$ we set
      \begin{align*}
            \fsn{\alpha} 
         := \inf \biggl\{ \sum_{j = 1}^k |a_j| \cdot v(M_j)
            \biggm|\; & k \in \N,\  
                      a_1, \dots, a_k \in \R\setminus\{0\},\\
                    & M_1, \dots, M_k \in S,\\
                    & \text{$f_1 \colon M_1 \rightarrow X, \dots, f_k
                      \colon M_k \rightarrow X$ continuous}\\
                    & \text{with~$\sum_{j=1}^k a_j \cdot H_d(f_j;\R) \fclr{M_j} = \alpha$}
            \biggr\};
      \end{align*}
      we use the conventions~$r \cdot \infty := \infty$ for all~$r \in
      \R_{>0}$ and $\inf \emptyset := \infty$.
  \end{itemize}
\end{defi}

\begin{thm}[Generating functorial semi-norms]\label{thm:genfunsn}
  Let $d \in \N$, let $S \subset \Mfd_d$ be a subclass, and let
  $\fsn{\cdot}$ be the semi-norm associated with a functorial
  semi-norm $v \colon S \longrightarrow [0, \infty]$ on
  fundamental classes of oriented closed connected $d$-manifolds
  in~$S$ (see Definition~\ref{def:assocfunsn}).
  \begin{enumerate}
  \item \label{thm:genfunsnv}
    Then $\fsn{\cdot}$ is a functorial semi-norm on singular
    homology in degree~$d$, and for all oriented closed connected
    $d$-manifolds~$M$ in~$S$ we have
    \[ \fsn[big]{\fclr M}
       = v(M).
    \]
  \item \label{thm:genfunsnfinite}
    If $S = \Mfd_d$ and $v$ is finite, then so is~$\fsn{\cdot}$.
  \item
    The associated semi-norm~$\fsn{\cdot}$ is maximal in the following
    sense: If $\fsn{\cdot}\!'$ is a functorial semi-norm on singular
    homology in degree~$d$ that extends~$v$, then 
    $\fsn{\cdot}\!' \leq \fsn{\cdot}$.
  \end{enumerate}
\end{thm}
\begin{proof}
  A straightforward computation shows that $\fsn{\cdot}$ as defined in
  Definition~\ref{def:assocfunsn} is indeed a functorial semi-norm in
  degree~$d$. If $M$ is an oriented closed connected $d$-manifold in~$S$,
  then representing~$\fclr M$ by~$\id_M \colon M\longrightarrow M$
  shows that $\fsn {\fclr M} \leq v(M)$. On the other hand, $v(M) \leq
  \fsn {\fclr M}$ as we now show.  Let 
  \[ \fclr M = \sum_{j=1}^k a_j \cdot H_d(f_j;\R)\fclr{M_j}
             = \sum_{j=1}^k a_j \cdot \deg f_j \cdot \fclr M
  \]
  be a representation of~$\fclr M$ as in
  Definition~\ref{def:assocfunsn}; then $1 = \sum_{j=1}^k a_j \cdot
  \deg f_j$ and hence
  \begin{align*}
    v(M) \leq \sum_{j=1}^k |a_j| \cdot |\deg f_j| \cdot v(M) 
         \leq \sum_{j=1}^k |a_j| \cdot v(M_j)
  \end{align*}
  by functoriality of~$v$ on~$S$. This proves the first part.

  The second part follows from the fact that every real singular
  homology class of a path-connected space is an $\R$-linear
  combination of rational classes, which can -- up to a non-zero
  factor -- be represented by oriented closed connected manifolds
  (Corollary~\ref{cor:repspacesr}).
  
  The last part follows directly from the construction
  of~$\fsn{\cdot}$, the triangle inequality, and the definition of
  functoriality.
\end{proof}

For example, the $\ell^1$-semi-norm can be viewed as the functorial
semi-norm generated by simplicial volume:

\begin{prop}
  Let $d \in \N \setminus \{3\}$. Then on the category of connected
  finite \CW-complexes the functorial semi-norm on singular homology
  in degree~$d$ associated with the simplicial volume in dimension~$d$
  coincides with the $\ell^1$-semi-norm in degree~$d$.
\end{prop}
\begin{proof}
  Clearly, the statement holds in degree~$0$. In degree~$1$ the claim
  follows directly from the Hurewicz theorem.

  In degree~$2$, one has to understand the simplicial volume of
  surfaces and how singular homology classes in degree~$2$ can be
  represented by surfaces: If $S$ is an oriented closed connected
  surface of genus~$g \geq 1$, then~\cite[p.~9,
  Proposition~C.4.7]{vbc,benedettipetronio} 
  \[ \|S\| = 4 \cdot g - 4 = 2 \cdot |\chi(S)|;
  \] 
  combining this fact with Proposition~\ref{prop:surfacel1fsns}
  proves the claim in degree~$2$.

  Suppose now that the degree~$d$ is at least~$4$. In view of
  Theorem~\ref{thm:genfunsn}~(3), the functorial semi-norm~$\fsn{\cdot}$
  associated with the simplicial volume satisfies~\mbox{$\fsn{\cdot} \geq
  \|\cdot\|_1$}; thus, it suffices to prove the reverse inequality. 

  Let $X$ be a connected finite \CW-complex, and let $\alpha \in
  H_d(X;\R)$. We can write~$\alpha = \sum_{j=1}^k a_j \cdot \alpha_j$,
  where $\alpha_1, \dots, \alpha_k \in H_d(X;\R)$ are rational and
  $a_1, \dots, a_n \in \R$. For~$n \in \N$ we let~$\alpha^{(n)} :=
  \sum_{j =1}^k a_j^{(n)} \cdot \alpha_j$, where $(a_j^{(n)})_{n \in
    \N}$ is a sequence in~$\Q$ that approximates~$a_j$; by
  construction, the~$\alpha^{(n)}$ are rational and the triangle
  inequality shows that
  \[ \lim_{n \rightarrow \infty} \fsn{\alpha^{(n)} - \alpha} = 0
     \quad\text{and}\quad
     \lim_{n \rightarrow \infty} \| \alpha^{(n)} - \alpha\|_1 = 0.
  \]
  Therefore, it suffices to prove~$\fsn{\cdot} \leq \|\cdot\|_1$ for
  rational classes in~$H_d(X;\R)$.

  If $\alpha$ is rational, then by Theorem~\ref{thm:rephomology},
  there is an~$a
  \in \R \setminus\{0\}$, and a continuous map~$f \colon
  M\longrightarrow X$ from some oriented closed connected
  $d$-manifold~$M$ such that
  \[ a\cdot H_d(f;\R) \fclr M = \alpha 
  \]
  and such that in addition~$\pi_1(f) \colon \pi_1(M) \longrightarrow
  \pi_1(X)$ is an isomorphism. Applying the mapping theorem in bounded
  cohomology~\cite[Section~3.1, Theorem~4.3]{vbc,ivanov} (combined
  with the duality principle for the
  $\ell^1$-semi-norm~\cite[Corollary on p.~17]{vbc}) shows that
  $H_d(f;\R) \colon H_d(M;\R) \longrightarrow H_d(X;\R)$ is isometric
  with respect to the $\ell^1$-semi-norm. In particular,
  \[ \|\alpha\|_1 = |a| \cdot \bigl\| \fclr M \bigr\|_1 
                  = |a| \cdot \|M\| \geq \fsn{\alpha}
                  . \qedhere
  \]
\end{proof}

The surface semi-norm is also a semi-norm defined as in
Defintion~\ref{def:assocfunsn}:

\begin{prop}
  Let $d \in \N$, let $S \subset \Mfd_{2\cdot d}$ be the subclass of
  products of $d$~oriented closed connected surfaces, and let
  \begin{align*}
    v \colon S & \longrightarrow [0,\infty] \\
    M & \longmapsto |\chi(M)|.
  \end{align*}
  Then $v$ is a functorial semi-norm on fundamental classes of
  oriented closed connected manifolds in~$S$, and the functorial
  semi-norm on~$H_{2\cdot d}(\args;\R)$ associated with~$v$ is the
  surface semi-norm of Example~\ref{exa:ssn}.
\end{prop}
\begin{proof}
  That $v$ indeed is functorial can, for example, be seen via the
  simplicial volume, the proportionality principle for simplicial
  volume, and the multiplicativity of the Euler
  characteristic~\cite[p.~303, Corollary~6.5]{gromov,loehmeasure}.

  That the semi-norm associated with~$v$ and the surface semi-norm 
  coincide follows directly from the definitions.
\end{proof}

\begin{prop}\label{prop:ssnlowdeg} 
  The surface semi-norm~$\fsn \cdot _S$ vanishes on all singular
  homology classes of simply connected spaces of degree~$2$ or~$4$.
\end{prop}

\begin{proof} 
  Let $X$ be a simply connected topological space, and let $\alpha \in
  H_*(X;\R)$ be a homology class of degree~$2$ or~$4$.

  If $\alpha$ is of degree~$2$, then -- because $X$ is simply
  connected -- we have an isomorphism~$H_2(X;\Z) \cong
  \pi_2(X)$. Hence every integral homology class in degree~$2$ is
  represented by a map from the sphere~$S^2$. Using the universal
  coefficient theorem and the fact that $S^2$ admits self-maps of
  arbitrarily large degree, it follows that the surface semi-norm
  vanishes on~$H_2(X;\R)$.

  Let $\alpha$ now be of degree~$4$. In view of the triangle
  inequality, we can assume without loss of generality that $\alpha$
  is rational. Then by Corollary~\ref{cor:repspacesr} we can represent
  $\alpha$ as
  \[ a \cdot H_d(f;\R)\fclr M = \alpha,
  \]
  where $M$ is an oriented closed simply connected 
  $4$-manifold, $f \colon M\longrightarrow X$ is a continuous map, and
  $a \in \R \setminus \{0\}$. Moreover, the simply connected
  $4$-manifold~$M$ is dominated by a product~$S^1 \times S^1 \times
  S$, where $S$ is a suitable oriented closed connected
  surface~\cite[Proposition~7.1]{kotschickloeh}. Because $\chi(S^1
  \times S^1 \times S) =0$ it follows that $\fsn \alpha_S = 0$.
\end{proof}

Similarly to the definition of the surface semi-norm, we can also take
hyperbolic manifolds as building blocks of a functorial semi-norm:

\begin{exa}[The hyperbolic semi-norm]\label{exa:hypfsn}
  Let $d \in \N$, let $H \subset \Mfd_d$ be the subclass of all
  oriented closed connected smooth $d$-manifolds that admit a
  hyperbolic Riemannian metric. Then 
  \begin{align*}
    v \colon H & \longrightarrow [0,\infty] \\
    M & \longmapsto \vol(M)
  \end{align*}
  is well-defined and functorial (because the volume of hyperbolic
  manifolds can be expressed in terms of the simplicial
  volume~\cite[Section~0.3, Theorem~6.2]{vbc,thurston} and because the
  simplicial volume is functorial).  
  
  We point out that it is still an open problem whether every manifold
  can be dominated by a hyperbolic manifold~\cite[Conjecture~7.2]{kotschickloeh};
  so it is not known whether the functorial semi-norm
  on~$H_d(\args;\R)$ associated with~$v$ is finite.
\end{exa}

\begin{rem}[Generating functorial semi-norms via Poincar\'e spaces]
  \label{rem:poincarespace}
  Recall that a \emph{$\Q$-Poincar\'e space of formal dimension~$d$}
  is a connected \CW-complex~$X$ together with a homology class~$ [X]
  \in H_d(X;\Q)$, the \emph{fundamental class}, such that
  \[ \args \cap [X]
     \colon 
     H^*(X;\Q) \longrightarrow H_{*-d}(X;\Q)
  \]
  is an isomorphism. In particular, one can introduce the notion of
  mapping degree for continuous maps between $\Q$-Poincar\'e spaces of
  the same formal dimension.

  Similarly to Definition~\ref{def:assocfunsn} and
  Theorem~\ref{thm:genfunsn}, any functorial semi-norm on the
  fundamental classes of $\Q$-Poincar\'e complexes of a given
  dimension gives rise to an associated functorial semi-norm on
  singular homology in the given degree.
\end{rem}

\section{Functorial semi-norms (not) induced\\ from the singular chain complex}
\label{sec:chainfsn}

One source of functorial semi-norms on singular homology is the class
of functorial semi-norms on the singular chain complex: Let $d \in
\N$. A \emph{functorial semi-norm on the singular chain complex in
  degree~$d$} consists of a choice of a semi-norm~$\fsn{\cdot}$
on~$C_d(X;\R)$ for every topological space~$X$ such that the following
``functoriality'' holds: for all continuous maps~$f \colon X
\longrightarrow Y$ between topological spaces and all~$c \in
C_d(X;\R)$ we have
\[ \fsn[big]{C_d(f;\R)(c)} \leq \fsn c. 
\]
Such a functorial semi-norm on the singular chain complex is
\emph{finite} if all the semi-norms involved are finite
semi-norms. For example, the $\ell^1$-norm on the chain level
(Example~\ref{exa:l1seminorm}) is a finite functorial semi-norm on the
singular chain complex.

\begin{nonexa}[$\ell^p$-Semi-norms]\label{nonexa:lp}
  Let $d \in \N$, let $p \in (1,\infty]$, and let $|\cdot|_p$ be the
  $p$-norm on~$C_d(\args;\R)$ with respect to the (unordered) basis
  given by the set of all singular $d$-simplices. Then $|\cdot|_p$ is
  \emph{not} a functorial semi-norm on the singular chain complex in
  degree~$d$:  

  We consider $X := \{x, y\}$ with the discrete topology and~$f \colon
  X \longrightarrow X$ mapping both points to~$x$. Let $c := \sigma_x
  + \sigma_y \in C_d(X;\R)$, where $\sigma_x$ and $\sigma_y$ are the
  constant singular $d$-simplices mapping to~$x$ and~$y$
  respectively. Then 
  \[   \bigl| C_d(f;\R)(c) \bigr|_\infty
     = | 2 \cdot \sigma_x |_\infty
     = 2
     > 1 
     = | \sigma_x + \sigma_y |_\infty
     = |c|_\infty
     ,
  \]
  and for $p \in (1,\infty)$ we obtain
  \[    \bigl| C_d(f;\R)(c) \bigr|_p
      = | 2 \cdot \sigma_x |_p
      = 2
      > \sqrt[p]{1^p + 1^p}
      = |\sigma_x + \sigma_y |_p
      = |c|_p
  \]
  Hence $|\cdot|_p$ is \emph{not} functorial.
\end{nonexa}

Clearly, any [finite] functorial semi-norm on the singular chain
complex in degree~$d$ induces a [finite] functorial semi-norm on
singular homology in degree~$d$ by taking the infimum of the
semi-norms of cycles representing a given class. Notice that being
induced from a finite functorial semi-norm on the singular chain
complex is a rather strong condition:

\begin{prop}\label{prop:chainfunsn}
  Let $d \in \N$ and let $\fsn{\cdot}$ be a finite functorial
  semi-norm on the singular chain complex in degree~$d$. Then
  \[ \fsn{\cdot} \leq \fsn{ \id_{\Delta^d}} \cdot \|\cdot\|_1. 
  \]
\end{prop}
\begin{proof}
  Let $c = \sum_{j = 0}^k a_j \cdot \sigma_j \in C_d(X;\R)$ be a
  singular chain (in reduced form). Viewing $\id_{\Delta^d} \colon
  \Delta^d \longrightarrow \Delta^d$ as a singular $d$-simplex
  on~$\Delta^d$, functoriality of~$\fsn{\cdot}$ yields 
  \begin{align*}
    \fsn{c} \leq \sum_{j=0}^k |a_j| \cdot \fsn{\sigma_j \circ \id_{\Delta^d}}
            \leq \sum_{j=0}^k |a_j| \cdot \fsn{\id_{\Delta^d}}
            = \fsn{\id_{\Delta^d}} \cdot \|c\|_1 ,
  \end{align*}
  as desired.
\end{proof}

\begin{cor}\label{cor:gromovforfinchnfsn}
  In particular, because the $\ell^1$-semi-norm is trivial on simply
  connected spaces~\cite[Section~3.1, Theorem~2.4]{vbc, ivanov}, every
  functorial semi-norm on singular homology induced from a
  \emph{finite} functorial semi-norm on the singular chain complex is
  trivial on simply connected spaces.
\end{cor}

Concerning the converse question ``Which [finite] functorial
semi-norms on singular homology are induced from [finite] functorial
semi-norms on the singular chain complex?'', we prove in the
following:
\begin{itemize}
  \item 
    Every functorial semi-norm on singular homology is induced from
    some (in general infinite) functorial semi-norm on the singular
    chain complex (Proposition~\ref{prop:indfromchains});
  \item 
    There exist finite functorial semi-norms on singular homology that
    are not induced from a \emph{finite} functorial semi-norm on the
    singular chain complex (Theorem~\ref{thm:finnonchainsn}). 

    So, Corollary~\ref{cor:gromovforfinchnfsn} is not strong enough to
    answer Gromov's question
    (Question~\ref{q:gromov}\enumref{enum:finite}) in the positive for
    \emph{all} finite functorial semi-norms.
\end{itemize}

\begin{prop}\label{prop:indfromchains}
  Let $d \in \N$, and let $\fsn{\cdot}$ be a functorial semi-norm on
  singular homology in degree~$d$. Then there is a functorial
  semi-norm~$|\cdot|$ on the singular chain complex in degree~$d$
  inducing~$\fsn{\cdot}$: i.e., for all topological spaces~$X$ and
  all~$\alpha \in H_d(X;\R)$ we have
  \[ \fsn \alpha = \inf \bigl\{ |c| 
                        \bigm|  \text{$c \in C_d(X;\R)$ is a cycle
                                representing~$\alpha$}
                        \bigr\}.
  \]
\end{prop}
\begin{proof}
  Let $X$ be a topological space. We denote by~$i \colon Z_d(X;\R)
  \longrightarrow C_d(X;\R)$ and $p \colon Z_d(X;\R) \longrightarrow
  H_d(X;\R)$ the inclusion of the $d$-cycles and the projection
  onto the $d$-th homology group respectively. We
  define a semi-norm~$|\cdot|$ on~$C_d(X;\R)$ by setting
  \[ |\cdot| := i_* p^* \fsn{\cdot},
  \]
  where $i_*$ and $p^*$ are defined as follows:
  \begin{enumerate}
    \item
      \emph{Construction of~$p^* |\cdot|$:}
      Let $p \colon V \longrightarrow U$ be a surjective homomorphism
      of $\R$-vector spaces, and let $|\cdot|$ be a semi-norm
      on~$U$. Then
      \begin{align*}
        p^*|\cdot| \colon V 
          & \longrightarrow [0,\infty] \\
        x & \longmapsto |p(x)|
      \end{align*}
      is a semi-norm on~$V$ (this is a straightforward calculation). 
    \item 
      \emph{Construction of~$i_*|\cdot|$:}
      Let $i \colon U \longrightarrow V$ be the inclusion of a
      subspace of an $\R$-vector space, and let $|\cdot|$ be a
      semi-norm on~$U$. Then 
      \begin{align*}
        i_*|\cdot| \colon V 
          & \longrightarrow [0,\infty] \\
        x & \longmapsto     
            \begin{cases}
              |x|    & \text{if $x \in U$},\\
              \infty & \text{if $x \in V \setminus U$}
            \end{cases}
      \end{align*}
      is a semi-norm on~$V$; clearly, $i_*|0| = |0| = 0$, and
      $i_*|\cdot|$ is compatible with scalar multiplication. Moreover,
      the triangle inequality is satisfied: Let $x$,~$y \in V$. If $x
      \in V \setminus U$ or $y \in V \setminus U$, then $i_*|x| =
      \infty$ or $i_*|y| = \infty$, so that the triangle inequality is
      trivially satisfied. The only remaining case is that $x$,~$y\in
      U$, and in this case the triangle inequality is satisfied,
      because $|\cdot|$ is a semi-norm on~$U$.

      Note that if $U \neq V$, then $i_*|\cdot|$ is infinite.
  \end{enumerate}
  Why is $|\cdot| = i_* p^* \fsn{\cdot}$ functorial? Let $f \colon X
  \longrightarrow Y$ be a continuous map and let $c \in C_d(X;\R)$. If
  $c$ is not a cycle, then $|c| = \infty$, and so $|C_d(f;\R)(c)| \leq
  |c|$. In case $c$ is a cycle, then $C_d(f;\R)(c)$ is a cycle as well
  and thus
  \[   \bigl| C_d(f;\R)(c)\bigr|
     = \fsn[big]{[C_d(f;\R)(c)]}
     = \fsn[big]{H_d(f;\R)[c]}
     \leq \fsn{[c]}
     = |c|
  \]
  because $\fsn{\cdot}$ is functorial.

  Moreover, $|\cdot|$ induces~$\fsn{\cdot}$ on homology because for
  all cycles~$c$ we have $\fsn{[c]} = |c|$ by construction of~$|\cdot|$.
\end{proof}

However, even if the given functorial semi-norm on singular homology is finite, the corresponding functorial semi-norm on the
singular chain complex provided in the proof of
Proposition~\ref{prop:indfromchains} is \emph{not} finite. This is not
merely an artefact of this construction: in the following we give an
example of a finite functorial semi-norm on singular homology that
grows too fast (compared to the $\ell^1$-semi-norm) to be induced from
a finite functorial semi-norm on the singular chain complex.

\begin{defi}[Degree monotonic map]
  A function~$\varphi \colon \R_{\geq 0} \longrightarrow \R_{\geq 0}$
  that is monotonically growing is called \emph{degree monotonic} if
  for all~$x \in \R_{\geq 0}$ and all~$d \in \N$ we have
  \[ \varphi(d \cdot x) \geq d \cdot \varphi(x).
  \]
\end{defi}

\begin{prop}\label{prop:degmoncomp}
  Let $d \in \N$ and let $v \colon \Mfd_d \longrightarrow \R_{\geq 0}$
  be  a finite functorial semi-norm on fundamental classes of oriented
  closed connected $d$-manifolds. If $\varphi \colon \R_{\geq0}
  \longrightarrow \R_{\geq0}$ is a degree monotonic map, then the
  composition 
  \[ \varphi \circ v \colon \Mfd_d \longrightarrow \R_{\geq0}
  \]
  is a finite functorial semi-norm on fundamental classes of oriented
  closed connected $d$-manifolds.
\end{prop}
\begin{proof}
  For all continuous maps~$f \colon M \longrightarrow N$ between
  oriented closed connected $d$-manifolds, we have~$v(M) \geq |\deg
  f| \cdot v(N)$, and thus
  \begin{align*}
      \varphi \circ v (M) 
    & \geq \varphi\bigl( |\deg f| \cdot v(N) \bigr)
      \geq |\deg f| \cdot \varphi \circ v (N)
  \end{align*}
  by the degree monotonicity of~$\varphi$.
\end{proof}

\begin{thm}\label{thm:finnonchainsn}
  There are finite functorial semi-norms on singular homology that are
  not induced from a finite functorial semi-norm on the singular chain
  complex.
\end{thm}
\begin{proof}
  Let $\varphi \colon \R_{\geq 0} \longrightarrow \R_{\geq 0}$ be a
  degree monotonic map that grows faster than linearly
  in the sense that $\lim_{x \rightarrow
    \infty} \varphi(x) / x = \infty$; 
  for instance, for every~$a \in \R_{> 1}$ the map
  \begin{align*}
    \R_{\geq 0} & \longrightarrow \R_{\geq 0} \\
    x           & \longmapsto     x^a
  \end{align*}
  has this property. Moreover, let $d \in \N_{\geq
    2}$.  

  We now consider the functorial semi-norm~$\fsn{\cdot}$ on
  singular homology in degree~$d$ associated with the finite
  functorial semi-norm on fundamental classes of oriented closed
  connected $d$-manifolds given by composing~$\varphi$ with the
  simplicial volume (Proposition~\ref{prop:degmoncomp} and
  Theorem~\ref{thm:genfunsn}); notice that $\fsn{\cdot}$ is finite.
  
  \emph{Assume} for a contradiction that $\fsn{\cdot}$ were induced
  from a finite functorial semi-norm. Then in view of
  Proposition~\ref{prop:chainfunsn} we would have
  \begin{align}\tag{$\mathbin{*}$}\label{eq:l1estimate} 
    \fsn{\cdot} \leq \fsn{\id_{\Delta^d}} \cdot \| \cdot \|_1.
  \end{align}
  However, we now show that $\fsn{\cdot}$ ``grows too fast'' to be
  able to satisfy this estimate. To see this consider the properties
  of hyperbolic manifolds more closely: Let $M$ be an oriented closed
  connected hyperbolic \mbox{$d$-mani}\-fold.  Then the fundamental
  group~$\pi_1(M)$ of~$M$ is residually
  finite~\cite[p.~542]{sambusetti}; so for any~$k \in \N$ there is a
  subgroup~$\Gamma_k \subset \pi_1(M)$ satisfying
  \[ k \leq \bigl[ \pi_1(M) : \Gamma_k \bigr] < \infty.
  \]
  For~$k \in \N$ we let $p_k \colon M_k \longrightarrow M$ denote the
  covering associated with the inclusion~$\Gamma_k \subset \pi_1(M)$;
  hence, $M_k$ also is an oriented closed connected (hyperbolic)
  $d$-manifold and
  \[ |\deg p_k| = \bigl[\pi_1(M) : \Gamma_k \bigr] \geq k. 
  \]
  Because $M$ is hyperbolic, the simplicial volume~$\| M\|$ is
  non-zero~\cite[Section~0.3, Theorem~6.2]{vbc,thurston}; thus,
  $\|M_k\| \geq k \cdot \|M\|$ tends to~$\infty$ for~$k \rightarrow
  \infty$. By definition, $\varphi$ grows faster than linearly and so
  \[ \frac{\fsn[big]{\fclr{M_k}}}
          {\bigl\|\fclr{M_k}\bigr\|_1} 
   = \frac{\varphi(\| M_k\|)}{\|M_k\|}
  \]
  tends to~$\infty$ for~$k \rightarrow \infty$, contradicting the
  estimate in Equation~\eqref{eq:l1estimate}. Therefore, the finite
  functorial semi-norm~$\fsn{\cdot}$ on singular homology  is not
  induced from a finite functorial semi-norm on the singular chain
  complex.
\end{proof}

\begin{question}
  In light of the example constructed in the proof of
  Theorem~\ref{thm:finnonchainsn}, it is natural to ask for a
  reasonable notion of equivalence of functorial semi-norms on
  singular homology or for a notion of domination of one functorial
  semi-norm by another. Is the $\ell^1$-semi-norm on singular homology
  ``maximal'' among finite functorial semi-norms on singular homology
  with respect to such a notion? (This should also be compared with
  Proposition~\ref{prop:equivstronglyinflexible}.)
\end{question}

\section{(In)flexible manifolds} \label{sec:(in)flexiblemanifolds}

The constructions of interesting functorial semi-norms 
in Section~\ref{subsec:nontrivialsimplyconn} below require as input
simply connected manifolds that are inflexible; recall that an
oriented closed connected manifold~$M$ is \emph{inflexible} if it
admits only self-maps of degree~$0$,~$1$ or~$-1$, i.e.,
$\deg(M,M) \subset \{-1, 0, 1\}$.

\begin{rem} 
  Looking at iterated compositions shows that an oriented closed
  connected manifold~$M$ is flexible if and only if $|\deg(M, M)| =
  \infty$.  Conversely, the manifold $M$ is inflexible if and only if
  $\deg(M, M)$ is finite.
\end{rem}

\begin{rem} 
  If a manifold is flexible, then -- by functoriality -- its
  simplicial volume is zero. In particular, oriented closed connected
  hyperbolic manifolds are inflexible, as they have non-zero
  simplicial volume. However, for simply connected manifolds the
  simplicial volume is zero and hence the simplicial volume cannot
  serve as an obstruction to flexibilty in this case.
\end{rem}

In this section we show how rational homotopy theory and surgery allow
one to construct examples of simply connected inflexible manifolds,
building upon examples of Arkowitz and Lupton~\cite{arkowitzlupton}
(Sections~\ref{subsec:strategy} and~\ref{subsec:inflexiblemanifolds}).
We briefly discuss strongly inflexible manifolds in
Section~\ref{subsec:stronglyinflexible}. Finally, in
Section~\ref{subsec:flexiblemanifolds}, we discuss the class of simply
connected flexible manifolds from the viewpoint of rational homotopy
theory.  To make this section more readable we have moved most of the
calculations with differential graded algebras and the proof of
inheritance properties of simply connected inflexible manifolds to the
appendices Section~\ref{app:I} and~\ref{app:II}.

\subsection{(In)Flexibility and rational homotopy theory} \label{subsec:strategy}

We start by giving an overview of the construction of simply connected
inflexible manifolds and introducing key notations and definitions.

Rational homotopy theory provides the rationalisation
functor~$\args_\Q$ on the category of simply connected spaces and an
equivalence of categories between the category of simply connected
rational spaces and the category of certain differential graded
algebras, the so-called minimal models. For the basic definitions in
rational homotopy theory, we refer to the book by F\'elix, Halperin
and Thomas~\cite{felixhalperinthomas}.

More concretely, if $M$ is an oriented closed simply connected manifold, then
the associated minimal model~$A_M$ is a differential graded algebra
over~$\Q$ whose cohomology coincides with the rational cohomology
of~$M$; in particular, $A_M$ has a cohomological fundamental
class~$[A_M]$, namely the cohomology class of~$H^*(A_M) \cong
H^*(M;\Q)$ dual to the fundamental class~$\fclq M$ of~$M$.  

Any self-map~$f \colon M \longrightarrow M$ induces a corresponding
dga endomorphism $A_f \colon A_M \longrightarrow A_M$; using the
cohomological fundamental class~$[A_M]$ of~$A_M$ we can associate a mapping
degree to~$A_f$, and this mapping degree coincides with~$\deg
(f)$. In particular, if $A_M$ is ``inflexible'', as defined in Definition \ref{defi:inflexiblealgebra/space} below,  then so is~$M$.

Hence, it suffices to find differential graded algebras that are
minimal models of simply connected manifolds and whose cohomological
fundamental class is inflexible; notice that the latter condition is algebraic by defintion and moreover that Theorem \ref{thm:bargesullivan} below entails that this is also true of the former condition.

We now give a precise definition of inflexibility and duality 
in the world of differential graded algebras:

\begin{defi}[(In)flexible (co)homology classes]\label{def:inflexibleclass}
  \hfil
  \begin{itemize}
    \item A homology class~$\alpha \in H_*(X;\Q)$ of a topological
      space~$X$ is \emph{flexible} if there is a continuous map~$f
      \colon X \longrightarrow X$ such that
      \[ H_*(f;\Q)(\alpha) = d \cdot \alpha 
      \]
      for some~$d \in \Q \setminus \{-1,0,1\}$.  A homology class is
      called \emph{inflexible} if it is not flexible.

      (In particular, an oriented closed connected manifold is
      inflexible if and only if its fundamental class is inflexible). 
    \item
      A cohomology class~$\alpha \in H^*(A)$ of a differential graded
      algebra~$A$ is \emph{flexible} if there is a dga endomorphism~$f
      \colon A \longrightarrow A$ such that
      \[ H^*(f) (\alpha) = d \cdot \alpha
      \]
      for some~$d \in \Q \setminus \{-1,0,1\}$. A cohomology class is
      \emph{inflexible} if it is not flexible.
  \end{itemize}
\end{defi}

\begin{defi}[Poincar\'e differential graded algebra]\label{def:poincaredga}
  Let $n \in \N$. A \emph{Poincar\'{e} differential graded algebra of
    formal dimension $n$} is a simply connected differential graded 
  algebra~$A$ together with a cohomology class~$[A] \in H^n(A)$, the
  \emph{fundamental class}, satisfying the following conditions:
  \begin{enumerate}
    \item For all~$j \in \N_{>n}$ we have~$H^j(A) = 0$.
    \item The map
      \begin{align*}
        \Q & \longrightarrow H^n(A) \\
        a & \longmapsto a \cdot [A]
      \end{align*}
      is an isomorphism. 
    \item For all~$j\in \{0,\dots,n\}$, the pairing~$H^j(A) \times
      H^{n-j}(A) \longrightarrow H^n(A) \cong \Q$ (where we use the
      isomorphism~$H^n(A) \cong \Q$ of the previous item) given by
      multiplication identifies $H^j(A)$ with~$\Hom_\Q(H^{n-j}(A), \Q)$.
  \end{enumerate}
\end{defi}

\begin{defi}[Inflexible Poincar\'e algebra/space] \label{defi:inflexiblealgebra/space}
  \hfil
  \begin{itemize}
    \item A Poincar\-e differential graded algebra~$(A,[A])$ is
      \emph{inflexible}, if its fundamental class~$[A]$ is inflexible
      in the sense of Definition~\ref{def:inflexibleclass}.
    \item A $\Q$-Poincar\'e space~$(X,[X])$ (see
      Remark~\ref{rem:poincarespace} for a definition) is
      \emph{inflexible}, if its fundamental class~$[X]$ is inflexible
      in the sense of Definition~\ref{def:inflexibleclass}. 
  \end{itemize}
\end{defi}


\subsection{Simply connected inflexible manifolds} \label{subsec:inflexiblemanifolds}

Arkowitz and Lupton gave examples of differential graded algebras that
admit only finitely many homotopy classes of dga
endomorphisms~\cite[Examples~5.1 and~5.2]{arkowitzlupton}. Moreover,
they showed how to prove that these differential graded algebras are minimal models
of simply connected closed manifolds. In particular, these simply
connected manifolds are inflexible. 

In Section~\ref{app:I} we review their construction, and give two
more examples of differential graded algebras with inflexible
fundamental class:

\begin{thm}
  There are inflexible Poincar\'e differential graded
  algebras~$(A_1,[A_1])$, $(A_2,[A_2])$, $(A_3,[A_3])$, $(A_4,[A_4])$ of
  formal dimensions~$64$,~$108$,~$208$, and $228$ respectively.
\end{thm}
\begin{proof}
  This is proved in Section~\ref{app:I}
  (Corollary~\ref{cor:poincaredga} and
  Proposition~\ref{prop:dgainflexible}), where also the choice of
  fundamental class is specified (Proposition~\ref{prop:dgafcl}).
\end{proof}

In the following, we focus on the realisability of these Poincar\'e
differential graded algebras by simply connected manifolds (for
simplicity, we consider only the case of trivial total Pontryagin class):

\begin{defi}[Realisability by manifolds]\label{def:realmfd}
  Suppose $(A,[A])$ is a Poincar\'e differential graded algebra of
  formal dimension~$n$.  We then write~$\realmfd {A,[A]}$ for the
  class of all oriented closed simply connected $n$-manifolds~$M$ that
  have trivial total Pontryagin class and that satisfy
  \[ (A_M,[A_M]) \cong (A,[A]). \]
\end{defi}

\begin{thm}[Simply connected inflexible manifolds]\label{thm:realisationbymfds}
  For the above Poincar\'e dgas $(A_1, [A_1]),
  \dots, (A_4,[A_4])$ the classes~$\realmfd{A_1,[A_1]}, \dots,
  \realmfd{A_4,[A_4]}$ are non-empty.  In particular, there are oriented closed simply
  connected inflexible manifolds of
  dimension~$64,\ 108,\ 208,\ 228$ respectively.
\end{thm}

We now assemble the statements we need to prove Theorem \ref{thm:realisationbymfds}.  As first step, we show that the differential graded algebras~$A_1,
\dots, A_4$ are the corresponding dgas of rational $\Q$-Poincar\'e
spaces:

\begin{prop}[Realisibility by $\Q$-Poincar\'e spaces]
  \label{prop:realisabilitypoincare}
  For the above Poincar\'e dgas $(A_1,[A_1]), \dots, (A_4, [A_4])$ there are
  corresponding simply connected rational $\Q$-Poincar\'e spaces $(X_1,[X_1]),
  \dots, (X_4,[X_4])$ respectively realising these dgas as their
  minimal models such that the cohomology classes corresponding to 
  the fundamental classes~$[A_j]$ are dual to the fundamental classes~$[X_j]$;
  these spaces~$X_1,\dots, X_4$ are unique up to rational homotopy
  equivalence, and they have formal dimension
  \[ 64,\ 108,\ 208,\ 228 
  \]
  respectively.
\end{prop}

\begin{proof}
  Because the dgas~$A_1, \dots, A_4$ are Poincar\'e, the
  correspondence between rational spaces and minimal Sullivan
  algebras~\cite[Chapter~17]{felixhalperinthomas} shows that up to rational homotopy equivalence there is a unique simply connected rational space that is a $\Q$-Poincar\'e space whose minimal model
  is~$A_1$, $A_2$, $A_3$ or~$A_4$ respectively, and whose fundamental
  class corresponds to the fundamental class of the respective dga.

  Moreover, there is a formula expressing the formal dimension in
  terms of the degrees of the generators of an elliptic
  dga~\cite[Proposition~38.3]{felixhalperinthomas}: The generators for
  our examples along with their degrees are given in Section
  \ref{subsec:dgadesign} and the calculation boils down to the formal
  dimension
  \[ |y_1| + |y_2| + |y_3| + |z|  
     - \bigl(|x_1| - 1\bigr) 
     - \bigl(|x_2| - 1\bigr), 
  \]
  and hence to the formal dimensions~$64$, $108$,
  $208$, and $228$ respectively.
\end{proof}

\begin{cor}[Inflexible $\Q$-Poincar\'e spaces]\label{cor:inflexiblepoincarespaces}
  In particular, the simply connected rational $\Q$-Poincar\'e
  spaces~$(X_1,[X_1]), \dots, (X_4,[X_4])$
  from Proposition~\ref{prop:realisabilitypoincare} are inflexible.
\end{cor}
\begin{proof}
  Let $j \in \{1,\dots,4\}$. \emph{Assume} for a contradiction that
  $X_j$ is flexible. Then there is a continuous map~$f \colon X_j
  \longrightarrow X_j$ of degree~$d \notin \{-1,0,1\}$. The map~$f$
  induces a dga morphism~$A_j \longrightarrow A_j$ of degree $d$,
  because $A_j$ is a minimal model of~$X_j$.  However, this
  contradicts inflexibility of the dga~$A_j$ established in
  Proposition~\ref{prop:dgainflexible}.
\end{proof}

It now remains to show that the rational $\Q$-Poincar\'e spaces of
Corollary~\ref{cor:inflexiblepoincarespaces} can be realised by simply
connected manifolds. To this end, we apply a foundational theorem of
Barge~\cite{barge} and Sullivan~\cite{sullivan} (a special case is
Theorem~\ref{thm:bargesullivan} below).  This theorem gives necessary
and sufficient conditions for a rational $\Q$-Poincar\'{e} space $X$
to be realised by a manifold with prescribed rational Pontryagin
classes; moreover the conditions are formulated using only the
rational cohomology ring of $X$.  Before stating the theorem we recall
some basic terminology:

Let $\lambda \colon H \otimes H \longrightarrow \Q$ be a non-singular
symmetric bilinear form over a finite dimensional $\Q$-vector
space~$H$.  Recall that a \emph{Lagrangian} for $(H, \lambda)$ is a
subspace $L \subset H$ such that $\lambda|_{L \times L} = 0$ and $2
\cdot \rank(L) = \rank(H)$; the pair~$(H, \lambda)$ is called
\emph{metabolic} if it admits a Lagrangian.  The \emph{Witt group
  of~$\Q$}, denoted~$W_0(\Q)$, is the Grothendieck group of the monoid
of isomorphism classes of non-singular symmetric bilinear forms on
finite dimensional $\Q$-vector spaces under the operation of direct
sum and modulo the subgroup generated by differences of metabolic
forms~\cite[I~\S~7] {milnorhusemoller}.

If $(X, [X])$ is a $\Q$-Poincar\'{e} space of formal dimension~$4k$ 
then the cup-product followed by evaluation on $[X]$ defines a
non-singular symmetric bilinear form $(H^{2k}(X; \Q), \lambda_{[X]})$.
The \emph{Witt index of $(X, [X])$} is defined to be the equivalence class of
this form in the Witt group of $\Q$:
\[ \tau_{[X]} := \bigl[H^{2k}(X; \Q), \lambda_{[X]}\bigr] \in W_0(\Q) .\]
  
\begin{thm}[Realising rational $\Q$-Poincar\'e spaces by manifolds 
  {\cite[Th\'eor\`eme~1, Theorem~13.2]{barge,sullivan}}]\label{thm:bargesullivan}
  Suppose that $(X, [X])$ is a rational $\Q$-Poincar\'{e} complex of
  formal dimension~$4k$ and that $p_* \in H^{4*}(X; \Q)$ is a
  cohomology class with $p_0 = 1 \in H^0(X; \Q)$.  Then there is an
  oriented closed simply connected manifold~$(M, [M])$ with total
  Pontryagin class $p_M$ and a rational equivalence $f \colon M \longrightarrow
  X$ with $H_{4k}(f; \Q)(\fclq M) = [X]$ and $H^{4k}(f; \Q)(p_M) =
  p_*$ if and only if the following two conditions hold:
  \begin{enumerate}
  \item The Witt index $\tau_{[X]}$ of $(X, [X])$ lies in the image of
    the homomorphism $W_0(\Z) \longrightarrow W_0(\Q)$.
  \item There is an equality $\sign(X, [X]) = \langle L(p_*), [X]
    \rangle$ where $L(p_*)$ is the Hirzebruch $L$-class evaluated at
    $p_*$.
  \end{enumerate}
\end{thm}

\begin{prop}[Witt index] \label{prop:Wittindex}
  Let $j \in \{1,\dots, 4\}$, and let 
  $(X_j, [X_j])$ be the corresponding $\Q$-Poincar\'{e} space of
  Proposition~\ref{prop:realisabilitypoincare} (oriented by the choice of
  fundamental class in Proposition~\ref{prop:dgafcl}). Then $(X_j,
  [X_j])$ has trivial Witt index, i.e., $\tau_{[X_j]} = 0 \in W_0(\Q)$. 
  In particular, the signature~$\sign(X_j)$ of~$(X_j,[X_j])$ equals~$0$.
\end{prop}
\begin{proof}
  The result follows by explicit computation.  For example, the
  intersection form of~$(A_1, [x_2^{16}])$ and hence~$(X_1, [X_1])$ is
  computed in Proposition~\ref{prop:lambdaA1} where a basis for the middle cohomology is given.  With respect to this matrix, the intersection matrix is an
  element of $\GL(4,\Z)$ and has Lagrangian with basis $\{ [x_2w],
  [x_1^2w] \}$.  Similar calculations prove the proposition for~$A_2,
  A_3$ and~$A_4$.
\end{proof}

\begin{proof}[Proof of Theorem~\ref{thm:realisationbymfds}]
  Let $j \in \{1,\dots,4\}$, and let $(X_j,[X_j])$ be the
  simply connected rational $\Q$-Poincar\'e space provided by
  Proposition~\ref{prop:realisabilitypoincare}. In view of
  Proposition~\ref{prop:Wittindex}, the Witt index~$\tau_{[X_j]}$ lies in
  the image of the homomorphism~$W_0(\Z) \longrightarrow W_0(\Q)$;
  choosing~$p := 1 \in H^0(X_j;\Q) \subset H^*(X_j;\Q)$, we
  obtain
  \[ \sign(X_j,[X_j]) = 0 = \langle L(p), [X] \rangle. 
  \]
  Therefore, by Theorem~\ref{thm:bargesullivan}, there exists an
  oriented closed simply connected manifold~$(M,[M])$ rationally
  equivalent to~$(X_j,[X_j])$ with trivial Pontryagin class; because
  $(A_j,[A_j])$ is the minimal model of~$(X_j,[X_j])$, it follows that
  $M \in \realmfd{A_j,[A_j]}$.

  In particular, this manifold~$M$ is inflexible (using the same
  arguments as in the proof of
  Corollary~\ref{cor:inflexiblepoincarespaces}).
\end{proof}

\begin{rem}[Scaling the fundamental class]\label{rem:scaling}
  The results of Theorem~\ref{thm:realisationbymfds},
  Proposition~\ref{prop:realisabilitypoincare},
  Corollary~\ref{cor:inflexiblepoincarespaces}, and
  Proposition~\ref{prop:Wittindex} all hold if the fundamental classes
  of the respective dgas/Poincar\'e complexes are scaled by any
  non-zero rational number. The key point is that if $\lambda$ is a
  non-singular symmetric bilinear form on a finite dimensional
  $\Q$-vector space that is trivial in the Witt group~$W_0(\Q)$ and if
  $a \in \Q\setminus \{0\}$, then also $a \cdot \lambda$ is trivial in
  the Witt group (because any Lagrangian for~$\lambda$ also is a
  Lagrangian for~$a \cdot \lambda$). Notice that scalars with
  different absolute values lead to different homotopy types of simply
  connected inflexible manifolds in the same rational homotopy type
  (Proposition~\ref{prop:scaling}).
\end{rem}

Starting with the manifolds in~$\realmfd{A_1,[A_1]}, \dots,
\realmfd{A_4,[A_4]}$ we can construct many more simply connected
inflexible manifolds; a detailed discussion of these results is
deferred to Section~\ref{app:II}.

\subsection{Strongly inflexible manifolds}\label{subsec:stronglyinflexible}

A manifold~$M$ is inflexible if and only if the set~$\deg(M, M)$ is
finite.  More ambitiously we can ask that $\deg(N, M)$ is finite for
any oriented manifold~$N$ of the same dimension as~$M$.  This leads to
the notion of strongly inflexible manifolds:

\begin{defi}[Strongly inflexible manifold]\label{def:stronglyinflexible}
  We call an oriented closed connected $d$-dimensional manifold~$M$ 
  \emph{strongly inflexible} if for any oriented closed connected
  $d$-dimensional manifold~$N$ the set~$\deg(N,M)$ 
  is finite.
\end{defi}

Clearly, any strongly inflexible manifold is also inflexible.

\begin{exa}\label{exa:hypstronglyinflexible}
  The simplicial volume can be used to show that oriented closed
  connected hyperbolic manifolds~$M$ are strongly inflexible: If $N$
  is an oriented closed connected manifold of dimension~$\dim M$, then 
  \[ |\deg f| \leq \frac{\|N\|}{\|M\|} < \infty 
  \]
  for any map~$f \colon N \longrightarrow M$; notice that $\|M\| > 0$
  as $M$ is hyperbolic.
\end{exa}

Unfortunately, we do not know of any \emph{simply connected} manifolds
that are strongly inflexible.  As in the case of inflexible manifolds,
rational homotopy theory and the examples from
Section~\ref{subsec:inflexiblemanifolds} and Section~\ref{app:I} could
be a good starting point for seeking strongly inflexible manifolds.
However one sees that the necessary calculations, if they are
possible, would be significantly more complicated than in the case of
inflexible manifolds.

\begin{question}\label{q:genericinflexible}
  Is every ``random'' Poincar\'e differential graded algebra of high
  formal dimension (strongly) inflexible?
\end{question}

A small piece of evidence supporting a positive answer to
Question~\ref{q:genericinflexible} is the bordism result in
Proposition~\ref{prop:inflexibleandrationalbordism}.

\subsection{Flexible spaces and manifolds}\label{subsec:flexiblemanifolds}
Clearly, all spheres (of non-zero dimension) are flexible manifolds,
and products of oriented closed connected manifolds with flexible ones
are flexible again.  Further examples of flexible manifolds and spaces can
be obtained via rational homotopy theory: 

\begin{prop}[Simply-connected manifolds of low dimension are flexible]\label{prop:formalflexible}
  Oriented closed simply connected formal manifolds are flexible. In
  particular, all oriented closed simply connected manifolds of
  dimension~$6$ or less are flexible.
\end{prop}
\begin{proof}
  Formal oriented closed simply connected manifolds admit many
  self-maps of non-trivial degree~\cite{shiga} and so are flexible.
  Moreover, by a classical result in rational homotopy theory, all
  oriented closed simply connected manifolds of dimension at
  most~$6$ are formal~\cite[Proposition~4.6]{neisendorfermiller};
\end{proof}

A natural generalisation of formality of minimal models is being pure: 

\begin{defi}[Pure]
  A Sullivan algebra~$(\bigwedge V, d)$ is \emph{pure} if $V$ is
  finite dimensional and 
  \[ d|_{V^{\text{even}}} = 0
     \qquad\text{and}\qquad
     d(V^{\text{odd}}) \subset \bigwedge V^{\text{even}};
  \]
  here, $V^{\text{even}}$ and $V^{\text{odd}}$ denote the even and the
  odd part respectively of the graded vector space~$V$.
\end{defi}

\begin{prop}[Pure rational spaces are ``almost flexible'']\label{prop:pureflexible}
  Let $X$ be a rational space whose minimal model is pure. Then every 
  rational homology class of~$X$ in positive degree is a sum of
  flexible homology classes.
\end{prop}

\begin{proof}
  Let $A = (\bigwedge V,d)$ be the minimal model of~$X$.  In view of
  the equivalence of categories between the category of minimal
  Sullivan dgas (and homotopy classes of dga morphisms) and the
  category of rational spaces (and homotopy classes of continuous
  maps) it suffices to show that every cohomology class
  in~$H^*(\bigwedge V,d) \cong H_*(X;\Q)$ in positive degree is a sum
  of flexible cohomology classes (as defined in
  Definition~\ref{def:inflexibleclass}).

  Let $f \colon \bigwedge V \longrightarrow \bigwedge V$ be the
  algebra morphism uniquely determined by the maps
  \begin{align*}
    V^{\text{even}} & \longrightarrow V\\
    x & \longmapsto 2^{|x|} \cdot x,\\[.5em]
    V^{\text{odd}} & \longrightarrow V\\
    y & \longmapsto 2^{|y|-1} \cdot y.
  \end{align*}
  Using the fact that $(\bigwedge V,d)$ is pure, a straightforward
  computation shows that $f$ is compatible with~$d$: On the even part,
  the differential vanishes, and so $f \circ d|_{V^{\text{even}}} = 0
  = d \circ f|_{V^{\text{even}}}$. The differential of an odd
  element~$y \in V^{\text{odd}}$ of~$V$ is a sum of products of even
  elements of~$V$ whose degrees sum up to~$|y| -1$, and so
  \[ f \circ d (y) =  2^{|y|-1} \cdot dy = d \circ f(y).
  \]
  
  Because $A$ is pure, there is an additional grading on~$A$ given by
  the word length in~$V^{\text{odd}}$; more explicitly, $A =
  \bigoplus_{k \in \N} A_{[k]}$, where
  \[ A_{[k]} := \bigwedge V^{\text{even}} 
             \otimes \bigwedge\negthinspace{}^k\, V^{\text{odd}}
  \]
  for all~$k \in \N$~\cite[p.~435]{felixhalperinthomas}; notice that
  the differential~$d$ is homoegeneous of degree~$-1$ with respect to
  this grading and that $f(z) = 2^{|z|-k} \cdot z$ holds for all~$k
  \in \N$ and all~$z \in A_{[k]}$. 

  So the dga morphism~$f$ witnesses that every cohomology class
  in~$H^*(A)$ of non-zero degree that can be represented by a cocycle in
  one of the subspaces~$A_{[k]}$ is flexible. On the other hand, using
  the direct sum decomposition~$A = \bigoplus_{k \in \N} A_{[k]}$ and
  the fact that $d$ is homogeneous of degree~$-1$ one
  can easily check that every cohomology class in~$H^*(A)$ is a sum of
  cohomology classes represented by such cocycles.
\end{proof}

Flexibility as established in Proposition~\ref{prop:formalflexible}
and~\ref{prop:pureflexible} provides a means to prove the vanishing of
finite functorial semi-norms on certain classes
(Corollary~\ref{cor:fsnvanish456} and
~\ref{cor:fsnvanishpure}). Clearly, the same methods apply whenever
the minimal models allow for an approriate grading or weight
function. For simplicity, we restricted ourselves to the cases above.

\section{Functorial semi-norms on simply connected spaces}
\label{sec:fsnsimplyconnected}

In the following we discuss Gromov's question whether all functorial
semi-norms on singular homology are trivial on simply connected
spaces (Question~\ref{q:gromov}).

Here, a key r\^ole is played by simply connected inflexible
manifolds. Recall that an oriented closed connected manifold~$M$ is
\emph{inflexible} if
\[ \deg(M,M) \subset \{-1,0,1\}.
\]

We start with a construction of a functorial semi-norm that is not
trivial on all simply connected spaces
(Section~\ref{subsec:nontrivialsimplyconn}); on the other hand, we
show in Section~\ref{subsec:positiveanswer} that the finite case of
Gromov's question can be answered affirmatively in all dimensions~$d
\leq 6$.

\subsection{Functorial semi-norms that are non-trivial on certain simply
  connected spaces}\label{subsec:nontrivialsimplyconn}

Using the construction from Section~\ref{sec:genfuncsn} and simply
connected inflexible manifolds, we obtain a (possibly infinite)
functorial semi-norm that is non-trivial on simply connected spaces:

Recall that an oriented closed connected manifold~$N$ is said to
\emph{dominate} an oriented closed connected manifold~$M$ of the same
dimension if there exists a continuous map~$N \longrightarrow M$ of
non-zero degree.

\begin{defi}[Domination $\Mfd_d$-semi-norm associated with a $d$-manifold]
  Let $M$ be an oriented closed connected $d$-manifold. Then the 
  \emph{domination $\Mfd_d$-semi-norm~$v_M \colon \Mfd_d \longrightarrow
    [0,\infty]$ associated with~$M$} is defined by
  \begin{align*} 
     v_M (N) & :=
     \sup \bigl\{|d| \bigm| d\in\deg(N,M)\bigr\}\\ 
         & \phantom{:}= 
           \sup \bigl\{ |\deg f|
                \bigm| \text{$f \colon N \longrightarrow M$ continuous}
                \bigr\}
        \in [0,\infty]
  \end{align*}
  for all~$N \in \Mfd_d$.
\end{defi}

\begin{prop} \label{prop:vmfn}
  If $M$ is an oriented closed connected $d$-manifold, then the domination
  $\Mfd_d$-semi-norm~$v_M$ is functorial. 
\end{prop}
\begin{proof}
  This follows from the definition of the domination semi-norm and
  multiplicativity of the mapping degree.
\end{proof}

\begin{defi}[Domination semi-norm associated with a $d$-manifold]
  Let $M$ be an oriented closed connected $d$-manifold.  Then the
  \emph{domination semi-norm~$\fsn{\cdot}_M$ on singular homology of
    degree~$d$} is the semi-norm on singular homology in degree~$d$
  associated with~$v_M$ (see Definition \ref{def:assocfunsn}).  By
  Proposition \ref{prop:vmfn} and Theorem~\ref{thm:genfunsn},
  $\fsn{\cdot}_{\!\!M}$ is a functorial semi-norm on~$H_d(\args;\R)$.
\end{defi}

\begin{cor}\label{cor:inflexiblesn}
  If $M$ is a simply connected closed inflexible manifold, then the
  domination semi-norm~$\fsn{\cdot}_{\!\!M}$ is not zero or infinite on all simply connected
  spaces. Hence there are functorial semi-norms on singular
  homology that are not zero or infinite on all simply connected
  spaces.
\end{cor}
\begin{proof}
  Let $M$ be a simply connected closed inflexible manifold; 
  such a manifold exists by
  Theorem~\ref{thm:realisationbymfds} -- we can even find such
  manifolds in infinitely many different dimensions
  (Corollary~\ref{cor:manyinflexible}).  
  By Theorem~\ref{thm:genfunsn}~\eqref{thm:genfunsnv} we have
  %
  \[ \fsn{[M]_\R}_{\!\!M} = v_M(M) = 1 \not\in\{0,\infty\}. \]
  Here, $[M]_\R$ is of course the $\R$-fundamental class of the simply
  connected closed manifold~$M$.  
  We give a graphical description of the domination semi-norm associated to~$M$ in
  Figure~\ref{fig:nontrivsn}.
\end{proof}

  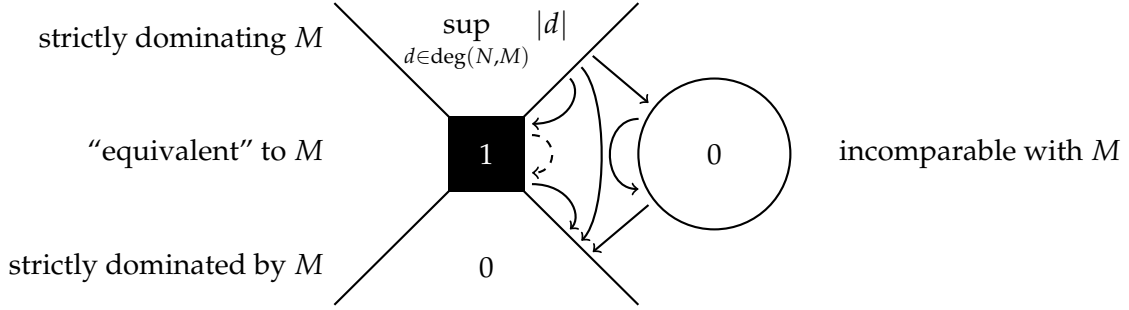
\begin{figure}
    \begin{center}
      \makebox[0pt]{%
      \def\fcolour{blue!40}
      \begin{tikzpicture}[x=0.5cm,y=0.5cm,thick]
        \draw (0,0) -- (-4,4);
        \draw (0,0) -- (4,4);
        \draw (0,3) node {$\displaystyle\sup_{d \in \deg(N,M)}|d|$};
        \draw (-4,3) node[left] {strictly dominating~$M$};
        \draw[->] (2.2,2) .. controls +(-45:0.5) and +(0:1) .. (1.2,0.8);
        \draw[->] (2.5,2.3) .. controls +(-45:1) and +(45:1) .. (2.5,-2.3);
        \draw (0,0) -- (-4,-4);
        \draw (0,0) -- (4,-4);
        \draw (0,-3) node {$0$};
        \draw (-4,-3) node[left] {strictly dominated by~$M$};
        \fill[color=black] (-1,-1) rectangle +(2,2);
        \draw[color=white] (0,0) node {\bf$1$};
        \draw (-4,0) node[left] {``equivalent'' to~$M$};
        \draw[dashed,->] (1.2,0.5) .. controls +(0:0.7) and +(0:0.7).. (1.2,-0.5);
        \draw[->] (1.2,-0.8) .. controls +(0:1) and +(45:0.5) .. (2.2,-2);
        \draw (6,0) circle(2);
        \draw (6,0) node {$0$};
        \draw (9,0) node[right] {incomparable with~$M$};
        \begin{scope}[shift={(6,0)}]
          \draw[->] (155:2.2) .. controls +(180:1) and +(180:1).. (205:2.2);
          \draw[->] (-142:2.2) -- (-3.2,-2.6);
          \draw[->] (-3.2,2.6) -- (142:2.2);
        \end{scope}
      \end{tikzpicture}}
    \end{center}
    \caption{Schematic construction of~$v_M$ in the proof of
      Corollary~\ref{cor:inflexiblesn}; the arrows indicate where maps
      of non-zero degree can exist, the dashed arrow indicates that only
      maps of degree~$-1$,~$0$,~$1$ can exist.}
    \label{fig:nontrivsn}
  \end{figure}

\begin{rem}
  We do not know whether the functorial semi-norms constructed in
  Corollary~\ref{cor:inflexiblesn} are finite.  If $M$ is a an oriented 
  closed connected $d$-manifold then by construction the
  domination $\Mfd_d$-semi-norm~$v_M$ is finite if and only if $M$ is strongly inflexible.
  It then follows by Theorem~\ref{thm:genfunsn}~\eqref{thm:genfunsnfinite} and the definitions
  that the associated functorial semi-norm
  $\fsn{\cdot}_{\!\!M}$ is finite if and only if $M$ is strongly inflexible.
  The existence of simply-connected strongly inflexible manifolds
  remains open at the time of writing.
\end{rem}

As we do not know of any simply connected strongly inflexible
manifold, Gromov's question
(Question~\ref{q:gromov}\enumref{enum:finite}) remains open for
\emph{finite} functorial semi-norms on singular homology.

\subsection{Partial results on finite functorial semi-norms on simply
  connected spaces}\label{subsec:positiveanswer}

In view of the Hurewicz theorem, all finite functorial semi-norms in
degree~$1$,~$2$ or~$3$ vanish on simply connected spaces: any integral
homology class of degree~$1$,~$2$ or~$3$ of a simply connected space
can be represented by a sphere. 

\begin{prop}\label{prop:equivstronglyinflexible}
  For $d \in \N_{\geq 4}$ the following statements are equivalent: 
  \begin{enumerate}
   \item There is a \emph{finite} functorial semi-norm~$\fsn \args$
     on~$H_d(\args;\R)$ such that for some homology class~$\alpha \in
     H_d(X;\R)$ of some simply connected space~$X$ we have $\fsn
     \alpha \neq 0$.
   \item There exists an oriented closed simply
     connected $d$-manifold that is strongly inflexible.
  \end{enumerate}
 \end{prop}

\begin{proof}
  First, let us assume that the first statement holds.  Without loss
  of generality we may assume that $X$ is path-connected and (in view 
  of the triangle inequality) that $\alpha$ is rational.  By
  Corollary~\ref{cor:repspacesr} we can write~$\alpha = a \cdot
  H_d(f;\R) \fclr M$, where $M$ is some oriented closed simply
  connected $d$-manifold, $f \colon M \longrightarrow X$ is a
  continuous map, and $a \in \R \setminus \{0\}$. We now show that the
  manifold~$M$ is strongly inflexible: Because $\fsn\args$ is finite
  and functorial, we obtain
  \[ \infty > \fsn[big]{\fclr M}
     \geq \frac1{|a|} \cdot |\alpha| > 0.
  \]
  If $N$ is an oriented closed connected $d$-manifold, then for all
  continuous maps~$g \colon N \longrightarrow M$ it follows that 
  \[ |\deg(g)| \leq \frac{\fsn[big]{\fclr N}}{\fsn[big]{\fclr M}} < \infty.  
  \]
  Hence, $\deg(N,M)$ is finite, and so $M$ is strongly inflexible. 

  Conversely, if there exists an oriented closed simply connected
  strongly inflexible $d$-manifold~$M$, we consider the functorial
  semi-norm~$\fsn \args$ associated with the domination
  semi-norm~$v_M$ for~$M$ (see
  Section~\ref{subsec:nontrivialsimplyconn}). Because $M$ is strongly
  inflexible, $v_M$ is finite. So by Theorem~\ref{thm:genfunsn}, also
  $\fsn\args$ is finite, and $\fsn{\fclr M} = v_M(M) =1 \not\in
  \{0,\infty\}$.
\end{proof}

\begin{cor}[Degrees~$4$,~$5$, and~$6$]\label{cor:fsnvanish456}
  All finite functorial semi-norms in degree~$4$,~$5$, and~$6$ vanish
  on simply connected spaces.
\end{cor}
\begin{proof}
  All oriented closed connected simply connected manifolds of
  dimension at most~$6$ are flexible
  (Proposition~\ref{prop:formalflexible}), and so cannot be strongly
  inflexible. Hence, the claim follows by applying
  Proposition~\ref{prop:equivstronglyinflexible}.
\end{proof}

Because finite functorial semi-norms vanish on flexible homology
classes, we obtain:

\begin{cor}\label{cor:fsnvanishpure}
  Let $X$ be a rational space whose minimal model is pure. Then every
  finite functorial semi-norm vanishes on every homology class of~$X$
  in positive degree.
\end{cor}
\begin{proof}
  This is a direct consequence of the fact that rational spaces with
  pure minimal model are almost flexible
  (Proposition~\ref{prop:pureflexible}). 
\end{proof}

Moreover, it follows from Gaifullin's work~\cite{gaifullin} that
finite functorial semi-norms which are multiplicative with respect to finite 
coverings are trivial on simply connected spaces:

\begin{defi}[URC-manifold~\protect{\cite[p.~1747]{gaifullinurc}}]
  Let $d \in \N$.
  An oriented closed connected $d$-manifold~$M$ is a \emph{URC-manifold} 
  (Universal Realisation of Cycles), if for every  topological space $X$ and every $\alpha \in H_d(X; \mathbb{Z})$,
  there is a finite sheeted covering $\overline M$ of $M$, a map $f \colon \overline M \to X$, 
  and~$k \in \Z \setminus \{0\}$ such that
  \[ H_d(f;\Z)\bigl([\overline M]\bigr) = k \cdot \alpha \in H_d(X;\Z).\]
\end{defi}

Gaifullin proved that there are many URC-manifolds in each
dimension~\cite[Theorem~1.3]{gaifullinurc}. Clearly, any URC-manifold
of dimension at least~$2$ is strongly inflexible and has non-zero
simplicial volume, because its finite coverings dominate hyperbolic
manifolds, which are strongly inflexible
(Example~\ref{exa:hypstronglyinflexible}).

\begin{exa}[Functorial semi-norms associated with coverings of URC-manifolds]\label{exa:stable}
  Let $d \in \N$, let $M$ be an oriented closed connected
  URC-mani\-fold of dimension~$d$, and let $S \subset \Mfd_d$ be the
  subclass of all finite connected covering spaces
  of~$M$. Then~$v_M|_S$ is a functorial semi-norm on~$S$. If $d
  \geq2$, then $M$ is strongly inflexible, and so $v_M|_S$ is a finite
  functorial semi-norm on~$S$ with~$v_M|_S(M) = 1$. More explicitly,
  multiplicativity under finite coverings and functoriality of 
  simplicial volume show that
  \[ v_M|_S(N) = \frac{\|N\|}{\|M\|} = \text{number of sheets of any covering~$N \rightarrow M$}
  \]
  holds for all~$N \in S$.

  Let $\fsn{\args}\!_M^c$ be the associated
  functorial semi-norm on~$H_d(\args;\R)$; because of the URC-property, 
  this functorial semi-norm~$\fsn\args\!_M^c$ is finite.
\end{exa}

\begin{prop}[Multiplicative finite functorial semi-norms]
  Let $d \in \N$ and let $\fsn\args$ be a finite functorial 
  semi-norm on~$H_d(\args;\R)$ that is multiplicative with respect to 
  finite coverings, i.e., satisfying: For all topological spaces~$X$, 
  all finite coverings~$p \colon Y \longrightarrow X$ and all~$\alpha \in H_d(Y;\R)$ 
  we have
  \[ \fsn[big]{H_d(p;\R) (\alpha)} 
     = \frac1k \cdot \fsn \alpha,
  \]
  where $k$ denotes the number of sheets of~$p$. Then there exists a
  constant~$c\in\R_{\geq 0}$ such that for all topological spaces~$X$
  and all~$\alpha \in H_d(X;\R)$ we have
  \[ \fsn \alpha \leq c \cdot \|\alpha\|_1. 
  \]
  In particular, $\fsn\args$ is trivial on simply connected spaces.
\end{prop}
\begin{proof}
  Without loss of generality we may assume that $d \geq 2$. Let $M$ be
  an oriented closed connected URC-manifold~$M$ in dimension~$d$. It
  follows from the arguments of Gaifullin that there is a constant~$a
  \in \R_{\geq 0}$ satisfying~\cite[Proposition~6.2]{gaifullin}
  \[ \fsn\args\!_M^c \leq a\cdot \|\args\|_1. 
  \]
  On the other hand, multiplicativity of~$\fsn\args$ and the
  construction of~$\fsn\args\!_M^c$ (Example~\ref{exa:stable}) show that
  \[ \fsn[big]{\fclr N} = \fsn[big]{\fclr M} \cdot \fsn[big]{\fclr N}\!_M^c 
  \]
  holds for all~$N \in S$ and hence that~$\fsn{\args} \leq \fsn{\fclr M} \cdot \fsn\args\!_M^c$. 
  Therefore, 
  \[ \fsn\args \leq a \cdot \fsn[big]{\fclr M} \cdot \|\args\|_1. \qedhere 
  \]
\end{proof}

\section{Appendix I: Four inflexible Poincar\'e dgas } \label{app:I}

This appendix is devoted to the algebraic side of inflexibility -- we
construct the four inflexible Poincar\'e differential graded algebras
used in Section~\ref{sec:(in)flexiblemanifolds}. We explain the
construction in Section~\ref{subsec:dgadesign}. In
Section~\ref{subsec:poincare}, we prove that these dgas are
Poincar\'e dgas; the intersection forms are computed in
Section~\ref{subsec:intersectionform}. In
Section~\ref{subsec:inflexibledga}, we show that these dgas are
inflexible. 
 
\subsection{A design pattern for possibly inflexible dgas} \label{subsec:dgadesign}
We start by defining a collection of dgas; all of the four concrete
examples below follow the same design pattern based on two examples of
Arkowitz and Lupton~\cite[Example~5.1 and~5.2]{arkowitzlupton}, which
are respectively examples $A_3$ and $A_4$ below.  We shall construct
dgas having the following properties:
\begin{itemize}
  \item two generators~$x_1$,~$x_2$ of even degree with trivial
    differential,
  \item four generators~$y_1$,~$y_2$,~$y_3$,~$z$ of odd degree; 
    the differential is given by
    \[
    \begin{array}[t]{lcl}
      dy_1 &:=& x_1^3 x_2\\
      dy_2 &:=& x_1^2 x_2^2\\
      dy_3 &:=& x_1 x_2^3,
    \end{array}
    \]
    and for the differential of~$z$ we choose~$z' \in \bigwedge
    (x_1,x_2,y_1,y_2,y_4)$ in such a way that  
    $d(y_1y_2y_3) = x_1^k \cdot z'$ or $d(y_1y_2y_3)= x_2^k
    \cdot z'$ and set 
    \[ dz := z' + x_1^{k_1} + x_2^{k_2}
    \]
    with suitable exponents~$k$,~$k_1$,~$k_2 \in \N_{>0}$.
\end{itemize}
By construction, $d \circ d(y_j) =0$ for all~$j \in \{1,2,3\}$ and $d
\circ d(z) =0$; moreover, these dgas are finitely generated minimal
dgas.

The following four examples dgas
  \[ A_j := \bigl(\bigwedge(x_1,x_2,y_1,y_2,y_3,z),d\bigr) \] 
with~$j \in \{1,\dots,4\}$ are all of this kind:

\begin{exa}[$A_1$: an elliptic inflexible dga of formal dimension~$64$]
  \label{exa:dga64}
  Define the dga $A_1$ with generators of  degrees
  \[ (|x_1|, |x_2|, |y_1|, |y_2|, |y_3|, |z|) = (2, 4, 9, 11, 13, 35) \] 
 where the differential~$d$ is given by
  \[ 
  \begin{array}[t]{lcl}
    dx_1 &:=& 0\\
    dx_2 &:=& 0
  \end{array}
  \quad
  \begin{array}[t]{lcl}
    dy_1 &:=& x_1^3 x_2\\
    dy_2 &:=& x_1^2 x_2^2\\
    dy_3 &:=& x_1 x_2^3
  \end{array}
  \quad
  \begin{array}[t]{lcl}
    dz &:=& x_2^4 y_1 y_2 - x_1 x_2^3 y_1y_3 + x_1^2x_2^2 y_2y_3\\
       & +& x_1^{18} + x_2^9\\
       & =& x_2^2 \cdot w + x_1^{18} + x_2^9,
  \end{array}
  \]
  where we use the abbreviation~$w := x_2^2 y_1y_2 - x_1x_2 y_1 y_3 +
  x_1^2 y_2y_3$; in other words~$x_1x_2w=
  d(y_1y_2y_3)$. 
\end{exa}

\begin{exa}[$A_2$: an elliptic inflexible dga of formal dimension~$108$]
  \label{exa:dga108}
    Define the dga $A_2$ with generators of  degrees
  \[ (|x_1|, |x_2|, |y_1|, |y_2|, |y_3|, |z|) = (4, 6, 17, 19, 21, 59) \] 
 where the differential~$d$ is given by
  \[ 
  \begin{array}[t]{lcl}
    dx_1 &:=& 0\\
    dx_2 &:=& 0
  \end{array}
  \quad
  \begin{array}[t]{lcl}
    dy_1 &:=& x_1^3 x_2\\
    dy_2 &:=& x_1^2 x_2^2\\
    dy_3 &:=& x_1 x_2^3
  \end{array}
  \quad
  \begin{array}[t]{lcl}
    dz &:=& x_2^4 y_1y_2 - x_1x_2^3 y_1y_3 + x_1^2x_2^2 y_2y_3\\
       & +& x_1^{15} + x_2^{10}.
  \end{array}
  \]
\end{exa}

\begin{exa}[$A_3$: an elliptic inflexible dga of formal
  dimension~$208$~\protect{\cite[Example~5.1]{arkowitzlupton}}] 
  \label{exa:dga208}
      Define the dga $A_3$ with generators of  degrees
  \[ (|x_1|, |x_2|, |y_1|, |y_2|, |y_3|, |z|) = (8, 10, 33, 35, 37, 119) \] 
 where the differential~$d$ is given by
  
  \[ 
  \begin{array}[t]{lcl}
    dx_1 &:=& 0\\
    dx_2 &:=& 0
  \end{array}
  \quad
  \begin{array}[t]{lcl}
    dy_1 &:=& x_1^3 x_2\\
    dy_2 &:=& x_1^2 x_2^2\\
    dy_3 &:=& x_1 x_2^3
  \end{array}
  \quad
  \begin{array}[t]{lcl}
    dz &:=& x_1^4x_2^2 y_1y_2 - x_1^5 y_1y_3 + x_1^6 y_2y_3\\
       & +& x_1^{15} + x_2^{12}.
  \end{array}
  \]
\end{exa}

\begin{exa}[$A_4$: an elliptic inflexible dga of formal
  dimension~$228$~\protect{\cite[Example~5.2]{arkowitzlupton}}] 
  \label{exa:dga228}
    Define the dga $A_4$ with generators of  degrees
  \[ (|x_1|, |x_2|, |y_1|, |y_2|, |y_3|, |z|) = (10, 12, 41, 43, 45, 119) \] 
 where the differential~$d$ is given by

  \[ 
  \begin{array}[t]{lcl}
    dx_1 &:=& 0\\
    dx_2 &:=& 0
  \end{array}
  \quad
  \begin{array}[t]{lcl}
    dy_1 &:=& x_1^3 x_2\\
    dy_2 &:=& x_1^2 x_2^2\\
    dy_3 &:=& x_1 x_2^3
  \end{array}
  \quad
  \begin{array}[t]{lcl}
    dz &:=& x_2^3 y_1y_2 - x_1x_2^2 y_1y_3 + x_1^2x_2 y_2y_3\\
       & +& x_1^{12} + x_2^{10}.
  \end{array}
  \]
\end{exa}

We will carry out the proofs in detail only for the dga~$A_1$ defined
in Example~\ref{exa:dga64} -- in fact, this is the most complicated of
the four examples and the other examples can be treated by analogous
arguments and calculations.

\subsection{The example dgas are Poincar\'e dgas}\label{subsec:poincare}
Recall that a minimal Sullivan algebra~$(\bigwedge V, d)$ is
called \emph{elliptic} if~$V$ and~$H^*(\bigwedge V,d)$ are finite
dimensional. 

\begin{prop}[Ellipticity]\label{prop:dgaelliptic}
  The dgas~$A_1$,~$A_2$,~$A_3$, and~$A_4$ are elliptic.
\end{prop}
\begin{proof}
  As these dgas are finitely generated by construction, it suffices to 
  show that their cohomology is finite dimensional. In other words, it 
  suffices to show that the cohomology is generated by nilpotent
  classes. Because the odd degree generators are nilpotent on the level of
  the dgas and because the differential is trivial on the even degree 
  generators, it suffices to show that the classes~$[x_1]$ and $[x_2]$
  are nilpotent. 

  We now show that $[x_1]$ and $[x_2]$ are nilpotent in~$H^*(A_1)$
  (the arguments for the other example dgas are similar): By
  definition of~$d$, we have 
  \begin{align*}
    [x_1]^{19} 
    & = [x_1^{19}]
      = \bigl[ x_1 dz - x_2 d(y_1y_2y_3) - x_1x_2^9 
        \bigr] 
      = \bigl[ d(x_1 z - x_2 y_1y_2y_3 - x_2y_3)
        \bigr]\\
    & = 0.
  \end{align*}
  Therefore we obtain
  \begin{align*}
    [x_2]^{18} 
    & = [x_2^9] \cdot [x_2^{9}] = \bigl([dz - x_2^2 w - x_1^{18}]\bigr)^2
      = \bigl( [x_2^2 w] - [x_1^{18}]
        \bigr)^2 \\
    & = \bigl[(x_2^2 w)^2\bigr] - 2 [x_1^{18}x_2^2 w] + [x_1]^{36}
      = 0 - 2 [d(x_1^{17}x_2 y_1y_2y_3)] + 0\\
    & = 0;
  \end{align*}
  notice that $w^2 = 0$ because every summand of~$w$ contains two of
  the three odd generators~$y_1$,~$y_2$,~$y_3$ and $y_j^2 = 0$.
\end{proof}

We will now select non-zero classes in the top cohomology, which will
play the r\^ole of fundamental classes:

\begin{prop}[Fundamental classes for~$A_1$, $A_2$, $A_3$, $A_4$]
  \label{prop:dgafcl}
  \hfil
  \begin{enumerate}
    \item The class~$[x_2]^{16}$ is non-zero in~$H^{64}(A_1)$.
    \item The class~$[x_2]^{18}$ is non-zero in~$H^{108}(A_2)$.
    \item The class~$[x_1]^{26}$ is non-zero in~$H^{208}(A_3)$.
    \item The class~$[x_2]^{19}$ is non-zero in~$H^{228}(A_4)$.
  \end{enumerate}
\end{prop}
\begin{proof}
  We give the proof only for~$A_1$, the other cases being similar. 
  \emph{Assume} for a contradiction that~$[x_2^{16}] = 0$
  in~$H^{64}(A_1)$; hence, there is an element~$u$ of~$A_1$ of
  degree~$63$ with~$du = x_2^{16}$. We can write~$u$ as
  \[ u = p z + p_{12} y_1y_2 z + p_{13}y_1y_3z + p_{23}y_2y_3 z 
  + p_1 y_1 + p_2 y_2 + p_3 y_3,
  \]
  where $p$, $p_{12}$, $p_{13}$, $p_{23}$, $p_1$, $p_2$, $p_3$ are
  homogeneous polynomials in~$x_1$,~$x_2$.

  Then 
  \begin{align*}
        x_2^{16} = du 
    & = p x_2^9 + p x_1^{18} + p x_2^4 y_1y_2 - p x_1x_2^3y_1y_3 + p
    x_1^2x_2^2 y_2y_3\\
    & \quad + p_{12} d(y_1y_2)z + p_{12} x_2^9 y_1y_2 + p_{12}x_1^{18}y_1y_2+ 0 
    \\ 
    & \quad + p_{13} d(y_1y_3)z + p_{13} x_2^9 y_1y_3 + p_{13}x_1^{18}y_1y_3+ 0 
    \\ 
    & \quad + p_{23} d(y_2y_3)z + p_{23} x_2^9 y_2y_3 + p_{23}x_1^{18}y_2y_3+ 0 
    \\ 
    & \quad + q,
  \end{align*}
  where $q$ is a homogeneous polynomial in~$x_1$,~$x_2$ that is
  divisible by~$x_1x_2$; the zeroes at the end of the lines stem from
  the fact that squares of odd degree elements are zero and each
  summand of~$w$ contains two of the three odd degree
  generators~$y_1$, $y_2$,~$y_3$.

  Because $A_1$ is freely generated by~$x_1, \dots, z$, comparing the
  $x_2^{16}$-coefficients on both sides shows that $p\neq0$. Moreover,
  comparing the $z$-coefficients of both sides yields
  \[ p_{12} d(y_1y_2) + p_{13}d(y_1y_3) + p_{23}d(y_2y_3) =0. 
  \]
  Comparing the coefficients of $y_1$,~$y_2$,~$y_3$ in this equation
  gives us 
  \begin{align*}
    - x_1^2x_2^2 p_{12} & = x_1x_2^3 p_{13} \\
    x_1^3x_2 p_{12}     & = x_1x_2^3 p_{23} \\
    x_1^3x_2 p_{13}     & = - x_1^2x_2^2 p_{23}.
  \end{align*}
  Because $\deg p_{12} = 8$, $\deg p_{13} = 6$, and $\deg p_{23} = 4$,
  a simple divisibility argument shows that there is an~$\eta \in \Q$
  with 
  \[ p_{13} = - \eta \cdot x_1 x_2, \qquad 
     p_{23} = \eta \cdot x_1^2, \qquad
     p_{12} = \eta \cdot x_2^2.
  \]
  Hence, comparing the summands of~$du$ that are divisible
  by~$y_1y_2$, but not by~$z$, shows that
  \[ 0 = p x_2^4 + p_{12} x_2^9 + p_{12} x_1^{18} 
       = p x_2^4 + \eta \cdot x_2^{11} + \eta \cdot x_1^{18} x_2^2.
  \]
  Because $p \neq 0$, it follows that $\eta =0$ (otherwise the last
  summand is not divisible by~$x_2^4$). 
  On the other hand, by an analogous argument, we obtain
  \[ 0 = -p x_1x_2^3 + p_{13}x_2^9 + p_{13} x_1^{18}
       = -p x_1x_2^3 - \eta \cdot x_1 x_2^{10} - \eta \cdot x_1^{19} x_2,
  \]
  and thus~$p =0$, contradicting~$p \neq 0$. So $x_2^{16}$ cannot be a
  coboundary.
\end{proof}

\begin{cor}[The dgas~$A_1, \dots, A_4$ are Poincar\'e dgas]
  \label{cor:poincaredga}
  The dgas~$A_1, \dots, A_4$ are Poincar\'e dgas with the cohomology
  classes in Proposition~\ref{prop:dgafcl} as fundamental classes.
\end{cor}
\begin{proof}
  The dgas $A_1, \dots, A_4$ are elliptic
  (Proposition~\ref{prop:dgaelliptic}). 
  By a classical result in rational homotopy
  theory~\cite[Proposition~38.3]{felixhalperinthomas}, cohomology
  algebras of elliptic minimal Sullivan algebras are Poincar\'e
  duality algebras; clearly, any non-zero cohomology class in the top
  cohomology can be chosen as fundamental class. 
\end{proof}

\subsection{The intersection forms of the example dgas}\label{subsec:intersectionform}

\begin{prop}[Intersection form of~$A_1$] \label{prop:lambdaA1}
  The classes~$[x_2w]$, $[x_1^2 w]$, $[x_1^{16}]$, and~$[x_2^8]$ form
  a $\Q$-basis of~$H^{32} (A_1)$ (the middle cohomology of~$A_1$), and
  the intersection form with respect to this basis and the fundamental
  class~$[x_2^{16}]$ of~$A_1$ (see Proposition~\ref{prop:dgafcl})
  looks as follows:
  \def\prflbl#1{%
    {\makebox[0pt][l]{\normalfont\scriptsize\color{red}\ref{dga#1}}}
  }
  \begin{align*}
    \begin{pmatrix}
      0^\prflbl{3} & 0^\prflbl{3} & 0^\prflbl{6} & -1^\prflbl{7}\;\; \\
      0^\prflbl{3} & 0^\prflbl{3} & 1^\prflbl{4} & 0^\prflbl{5} \\
      0^\prflbl{6} & 1^\prflbl{4} & 0^\prflbl{2} & 0^\prflbl{1} \\
     -1^\prflbl{7} & 0^\prflbl{5} & 0^\prflbl{1} & 1^\prflbl{0} 
    \end{pmatrix}
  \end{align*}
  (The superscripts in the matrix refer to the part of the proof where
  the corresponding coefficient is computed).
\end{prop}
\begin{proof}
  We first show that $H^{32}(A_1)$ is generated by~$[x_2 w]$, $[x_1^2
  w]$, $[x_1^{16}]$, and~$[x_2^8]$: What do cocycles of degree~$32$
  in~$A_1$ look like?  Clearly, $x_1^{16}$ and $x_2^8$ are cocycles of
  degree~$32$.  All cocycles in the subalgebra~$\bigwedge(x_1,x_2)$
  divisible by~$x_1x_2$ are in the image of~$d$ (by definition
  of~$dy_1$, $dy_2$, $dy_3$).  Because the differential is trivial
  on~$\bigwedge(x_1,x_2)$, it remains to look at cocycles of the form
  \[ u = p_{12} y_1y_2 + p_{13} y_1y_3 + p_{23} y_2y_3,
  \]
  where $p_{12}$,~$p_{13}$,~$p_{23} \in \bigwedge(x_1,x_2)$.
  If $du = 0$, then looking the coefficients of~$y_1$, $y_2$,
  and~$y_3$ respectively in~$du$ leads to 
  \begin{align*}
    x_1^3x_2 p_{12}   & =   x_1x_2^3 p_{23}\\
    x_1^2x_2^2 p_{12} & = - x_1x_2^3p_{13}\\
    x_1^3x_2p_{13}    & = - x_1^2x_2^2p_{23};
  \end{align*}
  As $\deg p_{12} = 12$, $\deg p_{13} = 10$, and~$\deg p_{23} = 8$, a
  simple divisibility consideration shows that there
  exist~$\eta_1$,~$\eta_2 \in \Q$ such that
  \begin{align*}
    p_{23} & = \eta_1 \cdot x_1^4 + \eta_2 \cdot x_1^2 x_2\\
    p_{12} & = \eta_1 \cdot x_1^2x_2^2 + \eta_2 \cdot x_2^3\\
    p_{13} & = -\eta_1 \cdot x_1^3x_2 - \eta_2 \cdot x_1 x_2^2;
  \end{align*}
  hence, $u = \eta_1 \cdot x_1^2 w + \eta_2 \cdot x_2 w$. So
  $H^{32}(A_1)$ indeed is generated as $\Q$-vector space by~$[x_2 w]$,
  $[x_1^2 w]$,$[x_1^{16}]$, and~$[x_2^8]$.

  As next step, we determine the corresponding matrix for the
  intersection form with respect to the fundamental class~$[x_2^{16}]$: 
  \begin{enumerate}
    \setcounter{enumi}{-1}
    \item\label{dga0} Because we chose~$[x_2^{16}]$ as fundamental
      class with respect to which the intersection form is computed,
      the matrix coefficient corresponding to column~$[x_2^8]$ and
      row~$[x_2^{8}]$ equals~$1$.
    \item\label{dga1} We have $[x_1^{16}] \cdot [x_2^8] = [d(x_1^{13}x_2^7 y_1)] =
      0$.
    \item\label{dga2} Moreover, $[x_1^{16}] \cdot [x_1^{16}] = [x_1^{32}] = 0$ as
      was shown in the proof of Proposition~\ref{prop:dgaelliptic}.
    \item\label{dga3}
      Squares of elements of~$A_1$ of odd degree are zero, and hence
      $w^2 =0$ (because each summand of~$w$ contains two of the three
      odd elements~$y_1$, $y_2$,~$y_3$). 
    \item\label{dga4}
      We have (because $(dz) w = 0 = ww$ as in the previous item)
      \begin{align*}
        [x_1^{16}]\cdot[ x_1^2 w] 
        & = [ (dz) \cdot w - x_2^9 w - x_2^2 w w]\\
        & = [ 0 - x_2^9 w - 0]\\
        & = - [ d(x_2^7 z) + x_2^{16} + x_2^7 x_1^{18}] \\
        & = - [x_2^{16}] + [d(x_1^{15}x_2^6 y_1)]\\
        & = - [x_2^{16}].
      \end{align*}
    \item\label{dga5}
      Moreover, $[x_2^{8}]\cdot[ x_1^2 w]
         = [d(x_2^7 x_1 y_1 y_2y_3)] = 0$.
    \item\label{dga6}
      Analogously, $[x_1^{16}] \cdot [x_2 w] = [d(x_1^{15}y_1y_2y_3)]
      = 0$.
    \item\label{dga7}
      Finally, 
      \begin{align*}
        [x_2^8] \cdot [x_2 w]
        & = [d (x_2^7 z) - x_1^{18}x_2^7 - x_2^{16}]\\
        & = [- d (x_1^{15} x_2^6 y_1) - x_2^{16}]\\
        & = - [x_2^{16}].
      \end{align*}
  \end{enumerate}

  Moreover, from the shape of this matrix we can easily deduce that
  the elements~$[x_2 w]$,
  $[x_1^2 w]$,$[x_1^{16}]$, and~$[x_2^8]$ are linearly independent
  over~$\Q$. 
\end{proof}

\begin{rem}[Intersection form of~$A_2$, $A_3$, $A_4$]
  Similarly to the previous proposition one can show that: 
  \begin{itemize}
    \item The classes~$[x_2^3y_1y_2 - x_1x_2^2 y_1y_3 + x_1^2x_2y_2y_3],
      [x_2]^8$ form a $\Q$-basis of the middle cohomology~$H^{54}(A_2)$
      of~$A_2$. The intersection form of~$A_2$ with respect to this
      basis and the fundamental class~$[x_2]^{18}$ of~$A_2$ is
      \[ \begin{pmatrix}
            0 & -1 \\
            -1 & 1
         \end{pmatrix}.
      \]
    \item The classes~$[x_1^2x_2^2 y_1 y_2 - x_1^3x_2
      y_1 y_3 + x_1^4 y_2y_3 ], [x_1]^{13}$ form a $\Q$-basis of the middle
      cohomology~$H^{104}(A_3)$ of~$A_3$. The intersection form
      of~$A_3$ with respect to this basis and the fundamental
      class~$[x_1]^{26}$ is
      \[\begin{pmatrix}
          0 & -1 \\
          -1 & 1
        \end{pmatrix}.
      \]
    \item The middle cohomology~$H^{114}(A_4)$ of~$A_4$ is zero.
  \end{itemize}
\end{rem}

\subsection{The example dgas are inflexible}\label{subsec:inflexibledga}
We now show that the four Poincar\'e dgas~$A_1$, $A_2$, $A_3$, and $A_4$ are
\emph{inflexible} in the sense that there is no dga morphism whose
induced homomorphism on cohomology maps the fundamental class to a
non-trivial multiple of itself.

\begin{prop}[Inflexibility]
  \label{prop:dgainflexible}
  The dgas~$A_1$, $A_2$, $A_3$, and $A_4$ are inflexible.
\end{prop}

\begin{proof}
  We will give the complete calculation only for the example~$A_1$; for
  the other dgas the calculation is similar, and even a bit simpler as
  the degrees of the even generators~$x_1$ and~$x_2$ are less
  entangled; moreover, for the dgas~$A_3$ and $A_4$ an argument is
  contained in the work of Arkowitz and Lupton~\cite[Examples~5.1
  and~5.2]{arkowitzlupton}. 

  Let $f \colon A_1 \longrightarrow A_1$ be a dga morphism; looking at
  the degrees of the generators of~$A_1$ we see that there
  are constants 
  $\alpha_1$, $\alpha_2$, $\alpha_{2,1}, \dots, \gamma$,~$\gamma_1 \in \Q$
  and homogenous polynomials~$p_1$,~$p_2$,~$p_3$ in~$x_1$,~$x_2$
  such that
  \begin{align*}
    f(x_1) & = \alpha_1 \cdot x_1 \\
    f(x_2) & = \alpha_2 \cdot x_2 + \alpha_{2,1} \cdot x_1^2\\
    f(y_1) & = \beta_1 \cdot y_1 \\
    f(y_2) & = \beta_2 \cdot y_2 + \beta_{2,1} \cdot x_1 y_1\\
    f(y_3) & = \beta_3 \cdot y_3 
             + \beta_{3,1} \cdot x_1^2 y_1 
             + \beta_{3,2} \cdot x_2 y_1
             + \beta_{3,3} \cdot x_1 y_2\\
    f(z)   & = \gamma \cdot z + \gamma_1 \cdot x_1 y_1 y_2 y_3 
             + p_1 y_2 + p_2 y_2 + p_3 y_3.
  \end{align*}
  Using that $f$ as a dga morphism is compatible with the
  differential~$d$ of~$A_1$ and that $A_1$ is freely generated
  by~$x_1, \dots, z$, we deduce constraints on the
  coefficients~$\alpha_1, \dots$. Notice that because we
  chose~$[x_2^{16}]$ as fundamental class of~$A_1$, we can read off the
  ``degree'' of~$f$ from the coefficient~$\alpha_2$, and it suffices
  to show that $\alpha_2 \in \{-1,0,1\}$.
  \begin{enumerate}
    \item\label{inflexy1} 
      Comparing the coefficients for~$f \circ d(y_1)$ and~$d \circ
      f(y_1)$, we obtain
      \[ \beta_1 = \alpha_1^3 \alpha_2
      \]
      and $\alpha_1^3 \alpha_{2,1} = 0.$ In particular, $\alpha_1 = 0$
      or~$\alpha_{2,1} =0$. 
    \item\label{inflexy2} 
      Comparing the coefficients for~$f \circ d(y_2)$ and~$d \circ
      f(y_2)$, we obtain in addition that
      \[ \beta_2 = \alpha_1^2 \alpha_2^2. 
      \]
    \item\label{inflexz}
      Moreover, we have 
      \begin{align*} 
        f \circ d (z)
        & 
          = \alpha_1^{18} \cdot x_1^{18}
            + (\alpha_2 \cdot x_2 + \alpha_{2,1} \cdot x_1^2)^9
            + \alpha_2^4 \beta_1\beta_2 \cdot x_2^4 y_1y_2
            + q, 
        \\
        d \circ f(z)
        & = \gamma \cdot x_1^{18} + \gamma \cdot x_2^9 + \gamma \cdot
        d(x_2 w) + \gamma_1 \cdot d(x_1 y_1y_2y_3)
        ,
      \end{align*}
      where $q \in (x_1x_2) \cdot A_1$. Comparing the coefficients of
      these elements shows that 
      \[ \alpha_1^{18} + \alpha_{2,1}^9 = \gamma = \alpha_2^9. 
      \]
      Because~$\gamma \cdot d(x_2w) + \gamma_1 \cdot d(x_1y_1y_2y_3)$
      and~$q$ are divisible by~$x_1x_2$, it follows that
      \[ \gamma = \alpha_2^4 \beta_1 \beta_2 = \alpha_2^7 \alpha_1^5 
      \]
      (in the second equation we used the results from
      Steps~\ref{inflexy1} and~\ref{inflexy2}).
  \end{enumerate}
  
  In view of Step~\ref{inflexy1} we can assume that $\alpha_1 =0$
  or~$\alpha_{2,1} =0$.  If $\alpha_1 =0$, then also $\alpha_2^9 =
  \gamma = \alpha_2^7 \alpha_1^5 =0$ by Step~\ref{inflexz}. On the other hand, if
  $\alpha_1 \neq 0$ and $\alpha_{2,1} = 0$, then 
  \[ \alpha_1^{18} = \gamma = \alpha_2^9
     \qquad
     \text{and}
     \qquad
     \alpha_2^7 \alpha_1^5 = \gamma = \alpha_2^9
  \]
  by Step~\ref{inflexz}. Now a small computation shows that $\alpha_2
  =1$. 
  Hence, $A_1$ is inflexible.
\end{proof}

\section{Appendix~II: more inflexible dgas and manifolds} \label{app:II}

In this appendix we produce more examples of inflexible manifolds from
the basic examples of Section~\ref{sec:(in)flexiblemanifolds}
and~\ref{app:I}: Using connected sums and products, we obtain in
infinitely many dimensions infinitely many homotopy types of oriented
closed simply connected inflexible manifolds (Section~\ref{subsec:sum}
and Section~\ref{subsec:product}). Moreover, we show that
inflexibility is ``generic'' in the sense that in infinitely many
dimensions, every oriented bordism class can be rationally represented
by a simply connected inflexible manifold
(Section~\ref{subsec:moreinflexiblemanifolds}).

Recall that if $(A,[A])$ is a Poincar\'e dga
(Definition~\ref{def:poincaredga}) then $\realmfd{A,[A]}$ denotes the
class of all oriented closed simply connected manifolds that have
trivial total Pontryagin class and realise this rational data
(Definiiton~\ref{def:realmfd}). 

\subsection{Inflexible connected sums}\label{subsec:sum}

In general, it is not clear that connected sums of inflexible
manifolds are inflexible; however, in certain cases inflexibility is
preserved under connected sums:

\begin{thm}[Inflexible connected sums]\label{thm:inflexiblesums}
  Let $M$ be an oriented closed simply connected $n$-manifold with
  inflexible minimal model and $\pi_{n-1}(M) \otimes \Q = 0$.  Suppose
  that $N_1, \dots, N_r$ are oriented closed simply connected $n$-manifolds
  such that $\deg(N_{j, \Q}, M_\Q)$ is finite for each $j \in
  \{1,\dots,r\}$. Then the iterated connected sum
  \[ M \csum N_1 \csum \cdots \csum N_r \] 
  is inflexible.  In particular, for all~$r \in \N$ the $r$-fold
  connected sum $M^{\# r}$ is inflexible.
\end{thm}

The proof of this theorem relies on applying repeatedly the following
lemma: Recall that $\deg(N, M)$
is the set of degrees of maps between oriented closed connected
manifolds~$N$ and~$M$; also, for subsets~$A$,~$B \subset \Z$ 
we write
\[ A + B :=  \{ a + b \mid a \in A,\ b \in B \} \subset \Z . 
\]

\begin{lem} \label{lem:connect-sum-degree}
  Let $N_1, N_2$ and $M$ be oriented closed simply connected
  $n$-manifolds with rationalisations $N_{1, \Q}$, $N_{2,\Q} $, and
  $M_\Q$.  If $\pi_{n-1}(M_\Q) = 0$ then
  \[ \deg(N_1 \csum N_2, M) \subset \deg(N_{1, \Q}, M_\Q) + \deg(N_{2, \Q}, M_\Q) \, . \]
\end{lem}

\begin{proof}
  Because rationalisation preserves rational cohomology, we have
  \[ \deg(N_1 \csum N_2, M) \subset \deg\bigl( (N_1 \csum N_2)_\Q , M_\Q \bigr); 
  \]
  so it suffices to show $\deg((N_1 \csum N_2)_\Q, M_\Q)
  \subset \deg(N_{1,\Q}, M_\Q) + \deg(N_{2,\Q}, M_\Q)$. To this end, we consider
  the cofibration sequence
  \begin{equation}\tag{$\mathbin{*}$} \label{eq:cofib}
    S^{n-1} \to N_1 \csum N_2 \to (N_1 \csum N_2) \cup_{S^{n-1}} D^n, 
  \end{equation}
  where we attach~$D^n$ along the inclusion~$i \colon S^{n-1} \to N_1
  \csum N_2$ where $S^{n-1}$ is the locus of the connected sum
  between~$N_1$ and~$N_2$.  Clearly, the space $W := ((N_1 \csum N_2)
  \cup_{S^{n-1}} D^n)$ is homotopic to the wedge $N_1 \vee N_2$: we
  will use this fact below.  From the cofibration
  sequence~\eqref{eq:cofib} and its rationalisation we obtain the
  following commutative diagram of exact sequences:
  \begin{equation}\tag{$\mathbin{**}$} \label{eq:Qconnectedsum}
    \xymatrix{
        [W, M] \ar[r] \ar[d]^{\cdot_\Q} 
      & [N_1 \csum N_2, M] \ar[r]^{} \ar[d]^{\cdot_\Q} 
      & [S^{n-1}, M] \ar[d]^{\cdot_\Q}
      \\
        [W_\Q, M_\Q] \ar[r] 
      & [(N_1 \csum N_2)_\Q, M_\Q] \ar[r]^{} 
      & [S^{n-1}_\Q, M_\Q] }
  \end{equation}
  The lower sequence can be seen to be exact by looking at a concrete
  description of Sullivan models of cell additions (up to
  quasi-isomorphism)~\cite[Diagram~13.15]{felixhalperinthomas}.

  But $[S^{n-1}_\Q, M_\Q] \cong \pi_{n-1}(M_\Q) = 0$ by assumption.
  Thus, up to homotopy every map from the connected sum $(N_1 \csum
  N_2)_\Q \to M_\Q$ factors through the map~$(N_1 \csum N_2)_\Q \to
  W_\Q$ induced by the inclusion.  We observed above that there is a
  homotopy equivalence $W \simeq N_1 \vee N_2$. The characterisation
  of rationalisations in terms of singular homology with integral
  coefficients~\cite[Theorems~9.3 and~9.6]{felixhalperinthomas}
  together with the Mayer-Vietoris sequence in homology show that
  $W_\Q \simeq N_{1,\Q} \vee N_{2,\Q}$.  Hence we have the equality
  \[ [W_\Q, M_\Q] = [N_{1,\Q}, M_\Q] \times [N_{2, \Q}, M_\Q]. 
  \] 
  So from the commutative diagram~\eqref{eq:Qconnectedsum} above we
  see that there is an inclusion~$\deg((N_1 \csum N_2)_\Q, M_\Q)
  )\subset \deg(N_{1, \Q}, M_\Q) + \deg(N_{2, \Q}, M_\Q)$, as desired.
\end{proof}

\begin{proof}[Proof of of Theorem \ref{thm:inflexiblesums}]
  For $N := M \csum N_1 \csum \cdots \csum N_r$, observe that the
  obvious collapse map~$N \to M$ has degree~$1$ and so $1 \in \deg(N,
  M)$.  Applying Lemma \ref{lem:connect-sum-degree} inductively we
  conclude that $\deg(N, M)$ is finite, since we have assumed that the
  sets $\deg(N_{j, \Q}, M_\Q)$ and $\deg(M_\Q, M_\Q)$ are finite.  But
  the monoid $\Map(N, N)$ of self maps of $N$ acts by pre-composition
  on the set $\Map(N, M)$ of maps from $N$ to $M$ and since $1 \in
  \deg(N, M)$ we see that there is an inclusion $\deg(N, N) \subset
  \deg(N, M)$.  Hence $\deg(N, N)$ is finite and $N$ is inflexible.
\end{proof}

In order to apply Theorem~\ref{thm:inflexiblesums} to our examples we
shall need information about the group~$\pi_{\dim M-1}(M) \tensor \Q$
for~$M \in \realmfd{A_j, [A_j]}$ with $j = \{1, \dots , 4\}$, where
$A_1,\dots,A_4$ are the dgas from Section~\ref{app:I}.  Recall that
$\pi_*(M) \otimes \Q$ is a $\Q$-vector space generated by the
indecomposable elements of the minimal model
of~$M$~\cite[Theorem~15.11]{felixhalperinthomas}.  Using the notation
of Section \ref{subsec:dgadesign}, it follows that
\[ \pi_*(M) \tensor \Q 
   \cong        \Q(x_1) \oplus \Q(x_2) 
         \oplus \Q(y_1) \oplus \Q(y_2) \oplus \Q(y_3) 
         \oplus \Q(z). 
\]
The degrees of the generators $x_1$,~$x_2$, $y_1$, $y_2$, $y_3$,~$z$
for each of the~$A_j$'s are listed in
Section~\ref{subsec:dgadesign}. In particular we obtain:

\begin{lem} \label{lem:homotopygroups} 
  For all $j \in \{1, \dots ,
  4\}$ and for all $M \in \realmfd{A_j, [A_j]}$, we have that
  $\pi_{n-1}(M) \tensor \Q = 0$ where $n$ is the dimension of $M$.
\end{lem}

Theorem \ref{thm:inflexiblesums} allows us to prove the existence of
large classes of inflexible manifolds.  We do this systematically
in Section \ref{subsec:moreinflexiblemanifolds} and for now present
the following simple example:

\begin{exa}
  Let $j \in \{1,\dots, 4\}$, let $M \in \realmfd{A_j,[A_j]}$, and let
  $r \in \N_{>0}$. Then the $r$-fold connected sum~$M^{\# r}$ is
  inflexible (and simply connected). Looking at the rational
  cohomology ring of these manifolds shows that $M^{\# r} \not\simeq
  M^{\# s}$, whenever~$r \neq s$.
\end{exa}

\subsection{Inflexible products}\label{subsec:product}

In general, it is not clear that products of inflexible manifolds are
inflexible as maps between products of manifolds cannot
necessarily be decomposed into maps on the factors; we will show now
that certain products of our basic examples of simply connected
manifolds are inflexible:

\begin{thm}[Inflexible products]\label{thm:inflexibleproducts}
  Let $j \in \{2,3,4\}$, let $M \in \realmfd{A_j, [A_j]}$ be a
  manifold as in Theorem~\ref{thm:realisationbymfds}, and let $k\in
  \N_{>0}$. Then the $k$-fold product~$M^{\times k}$ is inflexible
  (and simply connected).
\end{thm}

This result is proved in the following by carefully analysing the
algebraic counterpart, namely tensor products of the Poincar\'e
dgas~$A_2$,~$A_3$, and~$A_4$ respectively: Recall that given dgas~$A$
and~$B$ there is the tensor product dga~$A \tensor B$~\cite[Example~3
on p.~47]{felixhalperinthomas} and that $H^*(A \tensor B) \cong H^*(A)
\tensor H^*(B)$.  In particular, if $(A, [A])$ and $(B, [B])$ are
Poincar\'{e} dgas then so is the product~$(A \tensor B, [A] \tensor
[B])$.

\begin{prop}\label{prop:dgaprodinflexible}
  For each~$j \in \{ 2, 3, 4\}$ and for all~$k \in \N_{>0}$ the
  $k$-fold tensor product $(A_j^{\otimes k}, [A_j]^{\otimes k})$
  is an inflexible Poincar\'{e} dga.
\end{prop}

\begin{proof}
  We shall give the proof for $A_3$ and then state the modifications
  necessary for $A_2$ and $A_4$. Let us fix some notation: for an 
  index~$a \in \{1, \dots , k \}$ let $A_{3a}$ denote the 
  $a$-th copy of~$A_3$ in the $k$-fold product $A_{3}^{\otimes k}$.  
  Similarly, for generators $x_i, y_i \in A_3$ as in
  Section~\ref{subsec:dgadesign} let $x_{ia}$ and $y_{ia}$ denote the
  copy of~$x_i$ or~$y_i$ in~$A_{3a}$. Notice that because the
  fundamental class of each~$A_{3a}$ is given by~$[x_{1a}]^{26}$ the
  fundamental class of~$A_3^{\otimes k}$ is given by~$\otimes_{a=1}^k
  [x_{1a}]^{26}$. Therefore, we can read off the degree of dga
  endomorphisms of~$A_3^{\otimes k}$ by looking at the situation in
  degree~$|x_1|=8$. 

  Now let $f \colon A_{3}^{\otimes k} \to A_{3}^{\otimes k}$ be a dga
  endomorphism of non-zero degree. Since $A_3^{\otimes k}$ is
  Poincar\'{e} with finite dimensional cohomology it follows that $f$
  induces isomorphisms on all cohomology groups, and so $f$ is a dga
  isomorphism~\cite[Proposition 12.10(i)]{felixhalperinthomas}. Thus,
  $H^8(f) \colon H^8(A_3^{\otimes k}) \longrightarrow H^8(A_3^{\otimes
    k})$ is a $\Q$-linear isomorphism. By construction of~$A_3$, there
  is a canonical isomorphism~$(A_3^{\otimes k})^8 \cong
  H^8(A_3^{\otimes k})$, which identifies~$H^8(f)$
  with~$f|_{(A_3^{\otimes k})^8}$. In particular, also
  $f|_{(A_3^{\otimes k})^k}$ is a $\Q$-linear isomorphism.

  We shall show below that $H^8(f)$ is represented by a signed
  permutation matrix with respect to the obvious basis
  of~$(A_3^{\otimes k})^8 = (A_3^8)^{\oplus k}$. If this holds, then, 
  because $\otimes_{a=1}^k [x_{1a}]^{26}$
  is a fundamental class of~$A_3^{\otimes k}$, the dga map~$f$ has
  degree~$1$ or~$-1$, which proves that $A_3^{\times k}$ is
  inflexible. 
  
  In order to complete the proof it therefore remains to prove that
  $H^8(f)$ is represented by a signed permutation matrix: For each~$b
  \in \{1, \dots, k \}$ we have the dga projection~$p_b \colon
  A_{3}^{\otimes k} \to A_{3b}$ and the dga inclusion~$i_b \colon
  A_{3b} \to A_3^{\otimes k}$. Moreover, for~$a$,~$b \in
  \{1,\dots,k\}$ we consider the dga map
  \[ f_{ab} : = p_a \circ f \circ i_b \colon A_{3b} \to A_{3a} .
  \]
  Since $A_{3a} = A_3 = A_{3b}$, we have by
  Proposition~\ref{prop:dgainflexible} that $f_{ab}$ has
  degree~$0$,~$1$ or~$-1$. Because $[A_{3}] = [x_1]^{26}$ and $A_3^8
  =\Q \cdot x_1$ it follows that $f_{ab}(x_{1a}) = \pm x_{1b}$ or
  $f_{ab}(x_{1a}) = 0$.  Thus, for all $a \in\{ 1, \dots , k\}$ we
  obtain
  \begin{align*}\label{eq:fx1adescr}\tag{$\mathbin{*}$} 
    f(x_{1a}) = \sum_{b=1}^k \epsilon_{ab} \cdot x_{1b}, 
     \quad \text{where $\epsilon_{ab} \in \{-1, 0, 1\}. $} 
  \end{align*}
  We proceed now by contradiction: Suppose that for some~$a$ at least
  two of the coefficients~$\{\epsilon_{ab}\mid b \in \{1,\dots,k\}\}$
  are non-zero. By construction of~$A_3$ for~$i \in \{10, 33\}$ there are
  identifications~$(A_{3}^{\otimes k})^{i} 
     = \bigoplus_{a=1}^k A_{3a}^{i}$.  
  We now consider the equation
  \[ df(y_{1a}) = f(dy_{1a}) .
  \]
  The left hand side is a sum of monomials of the form~$x_{1c}^3
  x_{2c}$, which can be seen by looking at the definition
  of~$A_{3a}^{33}$ and of the differential on~$A_3$
  (Example~\ref{exa:dga208}). However, on the right hand side, we have
  $f(dy_{1a}) = f(x_{1a}^3x_{2a}) = f(x_{1a})^3 \cdot f(x_{2a})$.  
  Using the description of~$f(x_{1a})$ from
  Equation~\eqref{eq:fx1adescr} and the fact that there are two
  non-zero coefficients~$\epsilon_{ab}$ and~$\epsilon_{ab'}$, it
  follows that  the right hand side contains monomials of the form~$\pm
  C_{bb'c} \cdot x_{1b}^2 \cdot x_{1b'} \cdot x_{2c}$ where $b \neq
  b'$ and $C_{bb'c} \in \Q \setminus \{0\}$. But such monomials are
  not present on the left hand side, which is a
  contradiction. Therefore, we can conclude that for each~$a \in
  \{1,\dots,k\}$ only one of the coefficients~$\epsilon_{ab} \in \{-1,0,1\}$ is
  non-zero. As $H^8(f) = f|_{(A_3^{\otimes k})^8}$ is an isomorphism,
  it follows that $H^8(f)$ indeed is represented by a signed
  permutation matrix.

  For the dgas~$A_2$ and~$A_4$ the fundamental class is a power
  of~$x_2$ and so we repeat the line of argument this time using
  $(A_{2}^{\otimes k})^{10} = \oplus_{a=1}^k A_{2a}^{10}$ or
  $(A_{4}^{\otimes k})^{12} = \oplus_{a=1}^kA_{4a}^{12}$ and the
  equation $dy_{2a} = x_{1a}^2x_{2a}^2$ instead.
\end{proof}

Notice that the above proof does not directly carry over to the case
of the Poincar\'e dga~$(A_1,[x_2]^{16})$ because dga endomorphisms of~$A_1$
are slightly more complicated in degree~$|x_2| = 4$ than in the cases
discussed above.

\begin{proof}[Proof of Theorem~\ref{thm:inflexibleproducts}]
  The minimal model of $M^{\times k}$ is the $k$-fold tensor product
  $\smash{A_j^{\otimes k}}$~\cite[Example 1
  p.\,248]{felixhalperinthomas}; moreover, the fundamental class
  of~$\smash{M^k}$ corresponds to~$\smash{[A_j]^{\tensor k} \in
    A_j^{\otimes k}}$. Now the theorem follows because the Poincar\'e
  dga~$\smash{A_j^{\otimes k}}$ is inflexible by
  Proposition~\ref{prop:dgaprodinflexible}.
\end{proof}

\begin{cor}\label{cor:manyinflexible}
  In each of infinitely many dimensions there exist infinitely many
  rational homotopy types of oriented closed simply connected
  inflexible manifolds. 
\end{cor}
\begin{proof}
  Let $j \in \{2,3,4\}$, let $k \in \N_{>0}$, and let $r \in
  \N_{>0}$. Moreover, let $M \in \realmfd{A_j, [A_j]}$.
  Theorem~\ref{thm:inflexibleproducts} and
  Theorem~\ref{thm:inflexiblesums} (together with
  Lemma~\ref{lem:homotopygroups}) show that the oriented closed
  simply connected manifold~$(M^{\times k})^{\# r}$ is inflexible.
  The rational cohomology of these manifolds shows that
  if $r \neq r'$, then $(M^{\times k})^{\# r}$ and $(M^{\times k})^{\#
    r'}$ do not have the same rational homotopy type.
\end{proof}


\subsection{Evidence for the genericity of inflexibility} \label{subsec:moreinflexiblemanifolds}

In the following, we combine results of the preceding sections to
exhibit large numbers of simply connected inflexible manifolds: On the
one hand, we show that there are ``many'' homotopy types of simply
connected inflexible manifolds, and in particular that in many
dimensions simply connected manifolds are ``ge\-ner\-ic'' from the point
of view of oriented rational bordism. On the other hand, we show that simply
connected inflexible manifolds exist that satisfy tangential structure
constraints such as being parallelisable or non-spinable.

One way to create many (integral) homotopy types of simply connected inflexible
manifolds out of a single inflexible Poincar\'e dga is to rescale the
fundamental class of the dga in question: 

\begin{prop}[Scaling the fundamental class]\label{prop:scaling}
  Let $(A,[A])$ be an inflexible Poincar\'e dga, and let $a$,~$a' \in
  \Q \setminus \{0\}$ with~$|a| \neq |a'|$. If $M \in \realmfd{A,a \cdot
    [A]}$ and $M' \in \realmfd{A, a' \cdot [A]}$, then $M \not\simeq
  M'$.
\end{prop}

\begin{proof}
  Recall that any Poincar\'e dga is the minimal model of some simply
  connected rational $\Q$-Poincar\'e space (cf.~proof of
  Proposition~\ref{prop:realisabilitypoincare}); hence there is a
  rational $\Q$-Poin\-ca\-r\'e space~$(X,[X])$ realising~$(A,[A])$.

  Let $M \in \realmfd{A,a \cdot [A]}$ and $M' \in \realmfd{A, a' \cdot
    [A]}$; then the rationalisation of both~$M$ and~$M'$ coincides
  with~$X$, the only difference being that the fundamental classes are
  mapped to different multiples of~$[X]$. Let $\rho_M \colon M
  \longrightarrow M_\Q = X$ and $\rho_{M'} \colon M' \longrightarrow
  X$ be the canonical maps provided by the rationalisation construction;
  by definition, then
  \[ H_n(\rho_{M};\Q)\fclq M = a \cdot [X]
     \qquad\text{and}\qquad
     H_n(\rho_{M'};\Q)\fclq{M'} = a' \cdot [X],
  \]
  where $n:= \dim M = \dim M'$. \emph{Assume} for a contradiction that
  there is a homotopy equivalence~$f \colon M \longrightarrow M'$. By
  the universal property of
  rationalisation~\cite[Theorem~9.7(ii)]{felixhalperinthomas} there is
  a continuous map~$f_\Q \colon X \longrightarrow X$ with~$\rho_{M'}
  \circ f = f_\Q \circ \rho_M$. Hence,
  \begin{align*}
        \deg_{[X]} f_\Q \cdot a \cdot [X] 
    & = H_n(f_\Q \circ \rho_M;\Q) \fclq M\\
    & = H_n(\rho_{M'}\circ f;\Q) \fclq M 
      = \deg f \cdot a' \cdot [X].
  \end{align*}
  Because $f$ is a homotopy equivalence and because $X$ is inflexible,
  it follows that $|\deg f| = 1 = |\deg_{[X]} f_\Q|$. Therefore, $|a|
  = |a'|$, which is a contradiction. So $M \not \simeq M'$.
\end{proof}

\begin{exa}\label{exa:scaling}
  Let $j \in \{1,\dots,4\}$.  In view of Remark~\ref{rem:scaling}, for
  all scalars~$a \in \Q \setminus \{0\}$ the class~$\realmfd{A_j, a\cdot
    [A_j]}$ is non-empty. Therefore, by the proposition above, there
  are infinitely many homotopy types of oriented closed simply
  connected manifolds having the rational homotopy type given
  by~$A_j$; clearly, all of these manifolds are inflexible. 

  Similarly, for~$j \in \{2,3,4\}$ and all~$k \in \N_{>0}$ there are
  infinitely many homotopy types of oriented closed simply connected
  manifolds having the rational homotopy type given by~$A_j^{\otimes
    k}$ (because the corresponding Witt index is trivial as well, and
  so also the scalar multiples of the fundamental class are realisable
  by manifolds). 
\end{exa}

For Propositions \ref{lem:parallelrealise} and \ref{prop:nonspin}
below we shall need the follow lemma, which is a refinement of a
special case of the Barge-Sullivan Theorem~\ref{thm:bargesullivan}:

\begin{lem} \label{lem:parallelrealise} 
Let $(X, [X])$ be a
  $\Q$-Poincar\'{e} space of formal dimension~$4k$ with vanishing Witt index:~$\tau_{[X]} =
  0 \in W_0(\Q)$.  Then $(X,[X])$ can be realised by a stably parallelisable oriented closed simply
  connected smooth manifold.
  \end{lem}

\begin{proof}
The lemma follows from a little reflection upon the proof of the Barge-Sullivan theorem (Theorem~\ref{thm:bargesullivan}).  We need to find a stable bundle~$\xi$ over
 the rational space~$X$ such that the total Pontryagin class of~$\xi$
 is trivial; hence we may choose $\xi$ to be the trivial bundle.
 Since the manifold~$M$ produced by the Barge-Sullivan theorem
  has a normal map
  \[
  \xymatrix{
    \nu_M \ar[r] \ar[d] & \xi \ar[d] \\
    M \ar[r]^{\bar \nu} 
    & X }
  \]
  where $\nu_M$ is the stable normal bundle of~$M$, it follows that
  $M$ is stably parallelisable.
\end{proof}

\begin{cor} \label{cor:parallelexamples} For each of the example dgas
  $A_1, A_2, A_3$ and $A_4$ of Section \ref{subsec:dgadesign} and for
  each $a \in \Q \setminus \{ 0 \}$, the class $\realmfd{A_j,
    a\cdot  [A_j]}$ contains a stably parallelisable manifold.
  \end{cor}

\begin{proof}
  By Proposition~\ref{prop:Wittindex} the $\Q$-Poincar\'{e} spaces
  $(X_j, a \cdot [X_j])$ realising the Poi\-ncar\'{e} dgas $(A_j,  a\cdot [A_j])$ all
  have vanishing Witt index and so we may apply Lemma~\ref{lem:parallelrealise}.
\end{proof}

In light of Theorem~\ref{thm:inflexibleproducts} we introduce some notation: 
for~$j \in \{1,\dots, 4\}$ we write~$d_j$ for the formal
dimension of~$A_j$; more explicitly, $d_1 = 64$, $d_2=108$, $d_3 =
208$, $d_4=228$. Moreover, we abbreviate 
\begin{align*} 
  D &:= \{ d_1 \} \cup \{ d_j \cdot k \mid k \in \N_{>0}, j \in \{2,3,4\} \} \\
    & = \{64\} \cup \{ d \cdot k \mid k \in \N_{>0},\ d \in \{108,208,228\}\}.
\end{align*}

In dimensions in~$D$ we will now show that simply connected inflexible
manifolds are ``generic'' from the point of view of rational bordism,
thereby giving a first answer in the direction of
Question~\ref{q:genericinflexible}.

\begin{prop}[Inflexible manifolds and rational bordism]      
  \label{prop:inflexibleandrationalbordism}
  Let $n \in D$. Then there is a positive integer~$r(n)$, depending
  upon~$n$, such that for any oriented closed $n$-manifold~$N$ the
  $r(n)$-fold disjoint union~$\sqcup_{r(n)} N$, equivalently the
  $r$-fold connected sum~$\csum_{r(n)}N$, is oriented bordant to an
  oriented closed simply connected inflexible manifold.
\end{prop}

\begin{proof}
  Because the products of complex projective spaces form a $\Q$-basis
  of the rational bordism ring~$\Omega_*^{\SO} \otimes
  \Q$~\cite[Corollary 18.9]{milnorstasheff} and because the torsion
  in~$\Omega_*^{\SO}$ has exponent~$2$~\cite[Corollary
    1]{wall} it suffices to consider the case where $N$ is a product
  of complex projective spaces, say $N = \prod_{i=1}^m \C P^{n_i}$
  with~$2 \cdot (n_1 + \dots + n_m) = n$.

  By definition of~$D$, we can write $n = d_j \cdot k$, with $j \in
  \{2,3,4\}$ and $k \in \N_{>0}$, or~$j=1 = k$. Moreover,
  let $M \in \realmfd{A_j, [A_j]}$; by Lemma~\ref{lem:parallelrealise} we may assume that $M$ is
  stably parallelisable. Then $M^{\times k}$ is an oriented closed
  simply connected $n$-manifold that is inflexible (by
  Theorem~\ref{thm:inflexibleproducts}) and stably parallelisable. In
  particular, $M^{\times k}$ is oriented null-bordant.

  We now consider $N' := M^{\times k} \csum N$. By construction, $N'$
  is oriented bordant to~$N$ and simply connected. It hence suffices
  to show that $N'$ is inflexible: By Lemma~\ref{lem:homotopygroups},
  we have~$\pi_{n-1}(M^{\times k}) \otimes \Q\cong
  \pi_{n-1}(M)^{\times k} \otimes \Q= 0$. By definition,
  $H^2(M^{\times k};\Q) = 0$ if~$j > 1$; in the case~$n = d_1 =64$,
  there is no class~$x \in H^2(M;\Q)$ with~$x^{32} \neq 0$ (by
  definition, $H^2(M;\Q) \cong \Q \cdot x_1$, and $[x_1]^{32} =0$, as
  shown in the proof of Proposition~\ref{prop:dgaelliptic}).  However,
  there is a class~$x \in H^2(N;\Q) = H^2(\prod_{i=1}^m \C P^{n_i})$
  such that $x^{n/2}$ generates~$H^n(N;\Q)$. Therefore, $\deg(N,M^{\times k}) =
  \{0\}$, and now applying Theorem~\ref{thm:inflexiblesums} shows that
  $N' = M^{\times k} \csum N$ is inflexible.
\end{proof}

We saw above that there are many examples of stably parallelisable simply
connected inflexible manifolds.  On the other hand it is also possible
to find simply connected inflexible manifolds with other tangential
constraints.  For example we have:

\begin{prop}[Non-spinable inflexible manifolds] \label{prop:nonspin}
  For all~$n \in D$ there are oriented closed simply connected
  non-spinable inflexible manifolds of dimension~$n$.
\end{prop}  

\begin{proof}
  Let $N = S^{n-2} \tilde \times S^2$ be the total space of the sphere
  bundle of the non-trivial rank~$(n-1)$-vector bundle over~$S^2$.
  Then the second Stiefel-Whitney class of~$N$ generates~$H^2(N; \Z/2)
  = \Z/2$ and $N$ is non-spinable. 

  We write $n = d_j \cdot k$ with $j \in \{2,3,4\}$ and $k\in
  \N_{>0}$, or~$j = 1 = k$. Then for all~$M \in \realmfd{A_j,[A_j]}$
  the manifold~$M^{\times k}$ is inflexible (by
  Theorem~\ref{thm:inflexibleproducts}) and simply connected. So
  $M^{\times k} \csum N$ is non-spinable (because the Stiefel-Whitney
  class is non-trivial) and simply connected. We show now that
  $M^{\times k}\csum N$ is inflexible:

  As first step, we show that $\deg(N, M^{\times k}) = \{0\}$: A
  straightforward spectral sequence calculation shows that $H^d(N;\Q)
  = 0$ for all~$d \in \{4,6,8,12\}$. On the other hand, by
  construction of the Poincar\'e dgas~$A_1, \dots, A_4$ we have
  $H^d(M^{\times k};\Q) \neq 0$ for some~$d \in
  \{4,6,8,12\}$. Therefore, $\deg(N, M^{\times k}) = \{0\}$.

  Furthermore, from Lemma~\ref{lem:homotopygroups} we obtain
  $\pi_{n-1}(M^{\times k}) \otimes \Q =0$. Hence, $M^{\times k} \csum
  N$ is inflexible by Theorem~\ref{thm:inflexiblesums}.
\end{proof}


\enlargethispage{1cm}

\bigskip

\noindent
\begin{tabularx}{\linewidth}{XX}
\begin{minipage}{6cm}
\noindent
\emph{Diarmuid Crowley}\\
  {\small
  \begin{tabular}{@{\qquad}l}
    Institute of Mathematics\\
    University of Aberdeen\\
    Aberdeen AB24 3UE\\
    United Kingdom\\
    \textsf{dcrowley@abdn.ac.uk}\\
    \textsf{http://www.dcrowley.net}\\
  \end{tabular}}\end{minipage}
&
\begin{minipage}{6cm}
\noindent
\emph{Clara L\"oh}\\
  {\small
  \begin{tabular}{@{\qquad}l}
    Fakult\"at f\"ur Mathematik\\
    Universit\"at Regensburg\\
    93040 Regensburg\\
    Germany\\
    \textsf{clara.loeh@mathematik.uni-regensburg.de}\\
    \textsf{http://www.mathematik.uni-regensburg.de/loeh}
  \end{tabular}}\end{minipage}
\end{tabularx}
\end{document}